\documentclass[oneside,a4paper,reqno]{amsart}
\usepackage{pdfsync, enumerate}
\usepackage{stmaryrd}
\usepackage{mathrsfs}
\usepackage{hyperref}
\usepackage{tikz-cd}
\usepackage{amssymb}
\usepackage{stackrel}
\usepackage{multicol}
\usepackage{amsmath, amsthm, amscd, amssymb, latexsym, eucal}
\usepackage[all]{xy}
\def\serieslogo@{} \def\@setcopyright{} \makeatother

\usepackage{multienum}

\usepackage[colorinlistoftodos]{todonotes}

\usepackage{hyperref}
\usepackage{color}
\usepackage{cite}

\makeatletter
\renewcommand*\env@matrix[1][c]{\hskip -\arraycolsep
	\let\@ifnextchar\new@ifnextchar
	\array{*\c@MaxMatrixCols #1}}
\makeatother

\usepackage{color}

\numberwithin{equation}{section}
\newtheorem{thm}{Theorem}[section]
\newtheorem*{main-thm}{Theorem}
\newtheorem*{Auslander-thm}{Auslander's Theorem}

\newtheorem{cor}[thm]{Corollary}
\newtheorem{lem}[thm]{Lemma}
\newtheorem{prop}[thm]{Proposition}

\newtheorem*{thmA}{Theorem~A}
\newtheorem*{thmB}{Theorem~B}
\newtheorem*{thmC}{Theorem~C}

\newtheorem*{question1}{Question~1}
\newtheorem*{question2}{Question~2}
\newtheorem*{question3}{Question~3}

\theoremstyle{definition}
\newtheorem{defn}[thm]{Definition}

\newtheorem{rem}[thm]{Remark}
\newtheorem{exmp}[thm]{Example}

\newtheorem*{thmD}{Theorem~D}

\newtheorem*{problem}{Problem}

\newtheorem*{note}{Note}

\newcommand{\lxr}{\longrightarrow}

\DeclareMathOperator{\pd}{\mathsf{pdim}}

\DeclareMathOperator*{\gd}{\mathsf{gl.dim}}

\DeclareMathOperator*{\Findim}{\mathsf{Fin.dim}}

\DeclareMathOperator*{\Mod}{\mathsf{Mod}-\!}
\DeclareMathOperator*{\GrMod}{\mathsf{GrMod}-\!}
\DeclareMathOperator*{\Grmod}{\mathsf{Grmod}-\!}

\DeclareMathOperator*{\smod}{\mathsf{mod}-\!}

\DeclareMathOperator*{\lMod}{\!-\mathsf{Mod}}

\DeclareMathOperator*{\GrInj}{\mathsf{GrInj}-\!}

\DeclareMathOperator*{\Inj}{\mathsf{Inj}-\!}
\DeclareMathOperator*{\GrProj}{\mathsf{GrProj}-\!}

\DeclareMathOperator{\Ext}{\mathsf{Ext}}
\DeclareMathOperator*{\Tor}{\mathsf{Tor}}

\usepackage{stackengine}

\newsavebox{\proofbox}
\savebox{\proofbox}{\begin{picture}(7,7)%
	\put(0,0){\framebox(7,7){}}\end{picture}}

\usepackage{latexsym}
\usepackage{pstricks}
\usepackage{comment}

\usepackage[greek,english]{babel}

\newcommand{\gr}{\selectlanguage{greek}}

\begin{document}

\title{Injective generation for graded rings}

\author[Kostas]{Panagiotis Kostas}
\address{Department of Mathematics, Aristotle University of Thessaloniki, Thessaloniki 54124, Greece}
\email{pkostasg@math.auth.gr}

\author[Psaroudakis]{Chrysostomos Psaroudakis}
\address{Department of Mathematics, Aristotle University of Thessaloniki, Thessaloniki 54124, Greece}
\email{chpsaroud@math.auth.gr}

\date{\today}



\begin{abstract}
In this paper we investigate injective generation for graded rings. We first examine the relation between injective generation and graded injective generation for graded rings. We then reduce the study of injective generation for graded rings to the study of injective generation for certain Morita context rings and we provide sufficient conditions for injective generation of the latter. We then provide necessary and sufficient conditions so that injectives generate for tensor rings and for trivial extension rings. We provide two proofs for the class of tensor rings, one uses covering theory and the other uses the framework of cleft extensions of module categories. We finally prove injective generation for twisted tensor products of finite dimensional algebras.
\end{abstract}

\maketitle

\setcounter{tocdepth}{1} \tableofcontents

\section{Introduction and main results}

The finitistic dimension conjecture \textsf{FDC} is one of the leading problems of the so called ``Homological Conjectures" in the representation theory of finite dimensional algebras. It is a difficult problem to decide when a given algebra has finite finitistic dimension. In this direction, Rickard \cite{rickard} introduced the concept  of injective generation and proved that if injectives generate the derived category of a finite dimensional algebra, then the (big) finitistic dimension of the algebra is finite. By injective generation we mean that the smallest localizing subcategory generated by the injective modules coincides with the derived module category. This is always true for projective modules over any ring, but it can generally fail for injective modules (see \cite[Theorem 3.5]{rickard}). 
The concept of injective generation and its connection to the ``Homological Conjectures" was first communicated by Keller in 2001 \cite{keller}.

Rickard developed various techniques for showing that injective modules generate for some particular classes of algebras where the \textsf{FDC} holds. Motivated by Rickard's result, it is natural to investigate further classes of algebras where injective modules generate in order to deduce the \textsf{FDC}. In this paper we introduce and study injective generation for graded rings and graded injective modules, which we call {\em graded injective generation}. Our goal is to detect injective generation for rings with a graded structure and use the graded injective generation theory to provide new examples of algebras where injectives generate and therefore the \textsf{FDC} holds. Using the methods developed in this paper, we prove new reduction techniques for injective generation in the ungraded situation. The motivating question in the graded theory of injective generation can be simply formulated as follows: Let $R=\bigoplus_{i\in \mathbb{Z}} R_i$ be a positively and finitely graded ring over the integers. If injectives generate for $R_0$, does it follow that injectives generate for $R$? Using results of Cummings~\cite{cummings}, answering the above question positively is easy once injectives generate for $R_0$ and the projective dimension of $R_0$ as a left $R$-module is finite.

Based on the graded approach taken in this paper, we investigate injective generation for two well established classes of rings in representation theory: (i) tensor rings and (ii) trivial extension rings. The class of $\theta$-extension rings, denoted by $R\ltimes_{\theta}M$ and introduced by Marmaridis \cite{marmaridis}, provides a common framework for the aforementioned constructions. So, our principal question is when do injectives generate for the $\theta$-extension ring $R\ltimes_{\theta}M$? Our approach in this problem lies in finding concrete assumptions on $_RM_R$ so that the ring $R$ as a left module over $R\ltimes_{\theta}M$ has finite projective dimension.  We now get into the analysis of the main results of this paper. We start below with the concept of graded injective generation. 

Let $R=\oplus_{\gamma\in\Gamma}R_{\gamma}$ be a graded ring where $\Gamma$ is an abelian group. We denote by $\GrMod R$ the category of graded $R$-modules and by $\GrInj R$ the subcategory of graded injective modules. We say that \emph{graded injectives generate} for $R$ if $\mathsf{Loc}(\GrInj R)=\mathsf{D}(\GrMod R)$, 
where $\mathsf{Loc}(\GrInj R)$ denotes the smallest localizing subcategory generated by the graded injective modules (see Definition~\ref{defn_gradedinjgen}) and $\mathsf{D}(\GrMod R)$ denotes the derived category of graded $R$-modules. Rickard proved in \cite[Theorem 3.3]{rickard} that injectives generate for every commutative Noetherian ring. So it is natural to ask if the graded analogue of the latter result holds. In this direction, and using the class of $\epsilon$-commutative graded rings due to Dell'Ambrogio and Stevenson \cite{IvoStevenson}, we prove in Theorem~\ref{injgenepsilongraded} the following result. 

\begin{thmA}
Graded injectives generate for every $\epsilon$-commutative graded Noetherian ring. 
\end{thmA}

We continue by comparing injective generation with graded injective generation for a graded ring $R$. We first show that if injectives generate for $R$, then graded injectives generate for $R$ (see Proposition~\ref{ungradedimpliesgraded}). We show the converse in case that $R$ is a graded ring over a finite abelian group (see Proposition~\ref{finitegrading}). As a consequence we obtain in Corollary~\ref{injgenepsilongradednoeth} that injectives generate for any $\epsilon$-commutative graded Noetherian ring that is graded over a finite group. We then study injective generation for Morita context rings. We prove a reduction technique regarding Morita context rings with zero bimodule homomorphisms (see Theorem~\ref{injectivegenerationformorita}). In this way  and using also ladders of recollements of module categories \cite{ladder}, we formulate a result of Cummings \cite{cummings} on triangular matrix rings and injective generation.

We move now to the classes of examples that we aim to examine injective generation. It is a well known fact in representation theory that the category of graded modules $\GrMod R$ over a graded ring $R$ is equivalent to the category of modules $\Mod \Hat{R}$, where $\Hat{R}$ is the covering ring of $R$. By viewing $\Hat{R}$ as a Morita context ring (see Lemma \ref{coveringismorita}) and by using the theory built so far, we reduce the question on injective generation of a positively and finitely graded ring over the integers to the analogous question for Morita context rings with zero bimodule homomorphisms (Proposition~\ref{reduction}). Using the approach via coverings, and under certain conditions, we prove in Theorem~\ref{mainthm2} injective generation for tensor rings $T_R(M)$.  
  
\newpage  

\begin{thmB}
Let $R$ be a ring and $M$ a nilpotent $R$-bimodule such that: 
\begin{itemize}
\item[(i)] $\pd_RM<\infty$. 

\item[(ii)] $\mathsf{Tor}_i^R(M,M^{\otimes j})=0$ for all $i,j\geq 1$. 
\end{itemize}
Then injectives generate for $R$ if and only if injectives generate for $T_R(M)$. 
\end{thmB}

If a bimodule $_RM_R$ satisfies conditions (i) and (ii) above, we say that $M$ is \emph{left perfect}. We emphasise that the proof of the above result is based on a careful analysis of injective generation using covering theory and therefore using the graded approach taken in this paper. In another approach to systematically study the perfectness condition on the bimodule $M$ and its relation to injective generation, we investigate injective generation in the context of $\theta$-extension rings $R\ltimes_{\theta}M$. Note that for $\theta=0$ we get the well known class of trivial extension rings for which injective generation has been studied by Cummings, see \cite[Example 5.10]{cummings}. The module category of a $\theta$-extension ring gives rise to a cleft extension of module categories in the sense of Beligiannis \cite{beligiannis}. Under suitable conditions, we characterize in Theorem~\ref{injective_generation_for_cleft} injective generation in the framework of cleft extension of module categories. As a consequence, we show in Corollary~\ref{mainthm3} the following result. 

\begin{thmC}
Let $R$ be a ring, $M$ an $R$-bimodule and $\theta\colon M\otimes_RM\rightarrow M$ an associative $R$-bimodule homomorphism. If $M$ is left perfect and nilpotent, then injectives generate for $R$ if and only if injectives generate for $R\ltimes_{\theta}M$. 
\end{thmC}
  
The above result recovers Theorem~B but the techniques for proving Theorem~C are different from those used in Theorem~B. The proof of Theorem~C is based on the homological properties of a cleft extension under the presence of a perfect bimodule. In particular, we introduce the notion of a perfect endofunctor (Definition~\ref{perfect_functor}) which turns out to be the key homological property for proving Theorem~C. Finally, we investigate injective generation for twisted tensor products. Cummings proved in \cite[Proposition 3.2]{cummings} that if injectives generate for two finite dimensional algebras over a field, then the same holds for the tensor product algebra. We generalise this result in Theorem~\ref{injective_generation_for_twisted} by proving injective generation for twisted tensor product of algebras. This result relies on the graded injective generation theory of this paper. 

\begin{thmD}
Let $A$ and $B$ be finite dimensional algebras over a field $k$ that are also graded over abelian groups $\Gamma_1$ and $\Gamma_2$ respectively. If injectives generate for $A$ and $B$ then injectives generate for $A\otimes^tB$, where $t$ denotes a bicharacter $\Gamma_1\times\Gamma_2\rightarrow k^{\times}$.
\end{thmD}

The contents of the paper are organised as follows. Section~\ref{section:injgengraded} is divided into three subsections. In the first one we recall the necessary preliminaries of injective generation used in the sequel of the paper. In the second subsection we introduce graded injective generation and prove Theorem~A. In the third subsection we compare injective generation with graded injective generation. In Section~\ref{moritacontextrings} we first collect some basic facts about Morita context rings. We then continue by studying injective generation of Morita context rings using the theory of ladders of recollements of abelian categories \cite{ladder}. Moreover in Theorem~\ref{injectivegenerationformorita} we provide sufficient conditions for the injective generation of a Morita context ring with zero bimodule homomorphisms. In Section~\ref{section:covering} the reason for viewing covering rings as Morita context rings will be justified. In Lemma~\ref{lem:coveringzerozero} we prove for certain graded rings that the associated covering ring is a Morita context ring with zero bimodule morphisms. Also Proposition~\ref{reduction} offers useful properties of these Morita context rings with respect to injective generation.

Section~\ref{section:perfectbimod} is divided into four subsections. In the first subsection we study perfect bimodules and we provide sufficient conditions in Proposition~\ref{lemmataki} so that the corner algebras embedded as tuples in a Morita context ring have finite projective dimension. This is a crucial observation for Theorem~B and it relies on the perfectness assumption on the bimodule. In the second subsection some useful formulas for tensor products of modules over a Morita context ring are proved. In the third subsection we prove Theorem~B by showing first how to lift perfect bimodules on the associated Morita context ring of the covering ring, see Proposition~\ref{lifting_perfect_2}. In the fourth subsection, and as a byproduct of our results, we show in Corollary~\ref{cor: Beilinson algebra} injective generation for the Beilinson algebra. In Section~\ref{section:cleft extensions} we develop the necessary homological theory in a cleft extension of module categories in order to prove Theorem~C. In particular, we prove that if an endofunctor of a cleft extension of module categories is perfect and nilpotent (Definition~\ref{perfect_functor}) then the derived functor of a certain left adjoint vanishes in higher degrees, see Proposition~\ref{left_derived_of_q_vanishes}. The main result is proved in Corollary~\ref{mainthm3}. The final Section~\ref{section:twisted tensor products} is devoted to prove Theorem~D on injective generation of twisted tensor products. This section contains a subsection on open problems and further questions on injective generation of graded rings.

\medskip

\medskip

\subsection*{Acknowledgments}

The research project is implemented in the framework of H.F.R.I call ``Basic research Financing (Horizontal support of all Sciences)" under the National Recovery and Resilience Plan ``Greece 2.0" funded by the European Union -- NextGenerationEU (H.F.R.I. Project Number: 16785).

The results of this paper were presented by the first named author in the Algebra Seminar of the University of Stuttgart, the University of Bielefeld and the University of Padova (MALGA Seminar). The first author would like to express his gratitude to the members of the groups and especially to Steffen Koenig, Henning Krause and Jorge Vit\'oria, for the warm hospitality and the excellent working conditions. We would also like to thank the anonymous referee for valuable suggestions and comments.

\section{Injective generation for graded rings}
\label{section:injgengraded}

Injectives generate for a ring $R$ if the smallest localizing subcategory of $\mathsf{D}(R)$ that contains the injective $R$-modules coincides with $\mathsf{D}(R)$. In this section we introduce the analogous notion for graded rings, that we call \emph{graded injective generation} and compare the two. 

\subsection{Injective generation for rings} 

Throughout this paper, the triangulated categories that we consider have set indexed coproducts and all the coproducts under consideration are set indexed. 

\begin{defn}
\label{localizing}
Let $\mathcal{T}$ be a triangulated category. A \emph{localizing subcategory} of $\mathcal{T}$ is a triangulated subcategory of $\mathcal{T}$ that is closed under coproducts. Given a class $S$ of objects of $\mathcal{T}$, we denote by $\mathsf{Loc}_{\mathcal{T}}(S)$ the smallest localizing subcategory of $\mathcal{T}$ that contains $S$. In case $\mathsf{Loc}_{\mathcal{T}}(S)=\mathcal{T}$ we say that $S$ \emph{generates} $\mathcal{T}$. 
\end{defn}

We will usually write $\mathsf{Loc}(S)$ instead of $\mathsf{Loc}_{\mathcal{T}}(S)$, as the triangulated category $\mathcal{T}$ will be clear from context. We also recall the following, that will be used frequently.
\begin{defn} Let $\mathsf{F}\colon\mathcal{T}\to\mathcal{T}'$ be a triangle functor of triangulated categories.

\noindent (i) Let $\mathcal{D}$ be a triangulated subcategory of $\mathcal{T}'$. The \emph{preimage} of $\mathcal{D}$ under $\mathsf{F}$ is the smallest triangulated subcategory of $\mathcal{T}$ that consists of the objects $X$ such that $\mathsf{F}(X)$ is in $\mathcal{D}$. In case $\mathcal{D}=\mathcal{T}'$, we write $\mathsf{Preim}(\mathsf{F})$ for the preimage. 

\noindent (ii) Let $\mathcal{C}$ be a triangulated subcategory of $\mathcal{T}$. The \emph{image} of $\mathsf{F}$ applied to $\mathcal{C}$ is the smallest triangulated subcategory of $\mathcal{T}'$ that contains $\mathsf{F}(X)$ for all $X$ in $\mathcal{C}$. In case $\mathcal{C}=\mathcal{T}$, we write $\mathsf{Im}(\mathsf{F})$ for the image. 
\end{defn}

\begin{lem} \label{preimage_is_localizing}
Let $\mathsf{F}\colon\mathcal{T}\rightarrow\mathcal{T}'$ be a coproduct preserving triangle functor between triangulated categories. The preimage of any localizing subcategory of $\mathcal{T}'$ under $\mathsf{F}$ is a localizing subcategory of $\mathcal{T}$. 
\end{lem} 
\begin{proof}
By definition, the preimage of $\mathsf{F}$ is the smallest triangulated subcategory of $\mathcal{T}$ that contains the objects $X$ of $\mathcal{T}$ such that $\mathsf{F}(X)\in\mathcal{T}'$. If $\{X_i,i\in I\}$ is a set of objects in the preimage, this means that $\mathsf{F}(X_i)\in\mathcal{T}'$ for all $i\in I$, thus $\oplus \mathsf{F}(X_i)\in\mathcal{T}'$. However $\oplus\mathsf{F}(X_i)=\mathsf{F}(\oplus X_i)$, so indeed $\oplus X_i$ belongs in the preimage.
\end{proof}

\begin{prop}
\label{image_is_contained_in_localizing}
Let $\mathsf{F}\colon\mathcal{T}\to\mathcal{T}'$ be a triangle functor of triangulated categories and let $S$ and $S'$ be classes of objects in $\mathcal{T}$ and $\mathcal{T}'$ respectively. Suppose that $S$ generates $\mathcal{T}$, i.e.\ $\mathsf{Loc}(S) =\mathcal{T}$. If $\mathsf{F}$ preserves coproducts and $\mathsf{F}(s)$ lies in $\mathsf{Loc}(S')$ for every $s\in S$, then $\mathsf{Im}(\mathsf{F})$ is a subcategory of $\mathsf{Loc}(S')$. 
\end{prop}
\begin{proof}
The preimage of $\mathsf{Loc}(S')$ under $\mathsf{F}$ contains $S$, since objects of $S$ are mapped to objects of $S'$. Moreover, by Lemma \ref{preimage_is_localizing}, the preimage of $\mathsf{Loc}(S')$ is a localizing subcategory of $\mathcal{T}$, thus contains $\mathsf{Loc}(S)$. In case $\mathsf{Loc}(S)=\mathcal{T}$, we conclude that $\mathsf{Im}(\mathsf{F})$ is a subcategory of $\mathsf{Loc}(S') $.  
\end{proof}

Consider a ring $R$, which, for the rest of this paper, will always be assumed associative and unital and let $\Mod R$ denote the category of right $R$-modules and $\mathsf{D}(R)$ denote the derived category of $\Mod R$. Moreover, $R$-modules will be considered as objects of $\mathsf{D}(R)$, by viewing them as complexes in degree 0. We recall an important property of $\mathsf{D}(R)$ in the following remark. 

\begin{rem} \label{generation_by_projectives} 
    By \cite[Proposition 2.2]{rickard} the projective $R$-modules generate $\mathsf{D}(R)$. Moreover, since every projective module is a direct summand of a coproduct of copies of $R$, it follows that $R$ itself, as a right $R$-module, generates $\mathsf{D}(R)$, i.e $\mathsf{Loc}(R)=\mathsf{D}(R)$. 
\end{rem}

The following notion was considered by Keller \cite{keller} and recently studied by Rickard in \cite{rickard}: 

\begin{defn} \label{injective_generation_defn} 
We say that \emph{injectives generate} for a ring $R$ if $\mathsf{Loc}(\Inj R)=\mathsf{D}(R)$. 
\end{defn}

Let now $\Lambda$ be a finite dimensional algebra over a field. We write 
\[
\Findim\Lambda:=\sup\{\pd{X_{\Lambda}} \ | \ X\in \Mod{\Lambda} \ \text{and} \ \pd{X_{\Lambda}}<\infty \}
\]
for the big finitistic dimension of $\Lambda$. Rickard proved the following important result:

\begin{thm}
\textnormal{(\!\!\!\cite[Theorem~4.3]{rickard})}
Let $\Lambda$ be a finite dimensional algebra over a field. If injectives generate for $\Lambda$, then $\Findim\Lambda<\infty$. 
\end{thm}

We end this subsection with the following remark that will be used throughout. 

\begin{rem} 
\label{bounded_cohomology}
Let $R$, $S$ be two rings and assume that there is an adjoint pair $(\mathsf{F},\mathsf{G})$ of triangle functors between $\mathsf{D}(R)$ and $\mathsf{D}(S)$ as below 
\begin{center}
\begin{tikzcd}
\mathsf{D}(S) \arrow[rr, "\mathsf{G}"] &  & \mathsf{D}(R) \arrow[ll, "\mathsf{F}"', bend right]
\end{tikzcd}
\end{center}
If $\mathsf{F}$ maps complexes bounded in cohomology to complexes bounded in cohomology, it follows (see for instance \cite[Lemma 2.14]{cummings}) that $\mathsf{G}$ maps bounded complexes of injectives to bounded complexes of injectives. In particular, assume that there is a ring homomorphism $R\rightarrow S$ and consider the induced adjoint pair between the derived categories of $R$ and $S$
\[
\begin{tikzcd}
\mathsf{D}(S) \arrow[rr, "\mathsf{res}"] &  & \mathsf{D}(R) \arrow[ll, "-\otimes_{R}^{\mathsf{L}}S"', bend right]
\end{tikzcd}
\]
If $\mathsf{fdim}_RS$ (the flat dimension of $S$ as a left $R$-module) is finite, then for every $R$-module $X$, we have that $\mathsf{Tor}_n^R(X,S)=0$ for $n$ large enough. In other words, $X\otimes_R^{\mathsf{L}}S$ is a complex that is bounded in cohomology. We conclude that the functor $-\otimes_R^{\mathsf{L}}S$ maps complexes bounded in cohomology to complexes bounded in cohomology and therefore the functor $\mathsf{res}\colon\mathsf{D}(S)\rightarrow\mathsf{D}(R)$ maps bounded complexes of injectives to bounded complexes of injectives. 
\end{rem}

\subsection{Injective generation for graded rings} 

Throughout this subsection $\Gamma$ denotes an (additive) abelian group. We recall the following notions and fix the notation that follows: 

(i) A \emph{graded ring} over $\Gamma$ is a ring $R$ together with a decomposition $R=\oplus_{\gamma\in\Gamma}R_{\gamma}$ of abelian groups $R_{\gamma}$ such that $R_{\gamma}\cdot R_{\gamma'}\subseteq R_{\gamma+\gamma'}$ for all $\gamma,\gamma'\in\Gamma$. We note that necessarily $R_0$ is a subring of $R$ (that we will be calling the \emph{initial subring} of $R$) and for every $\gamma$, $R_{\gamma}$ is an $R_0$-bimodule. Moreover, elements that belong in some $R_{\gamma}$ are called \emph{homogeneous}. Lastly, for a graded ring $R$ over $\mathbb{Z}$, we say that $R$ is \emph{positively graded} if $R_i=0$ for all $i<0$ and that $R$ is \emph{finitely graded} if $R_i=0$ for almost all $i\in\mathbb{Z}$.   

(ii) A (right) \emph{graded $R$-module} is a (right) module $M$ together with a decomposition $M=\oplus_{\gamma\in\Gamma}M_{\gamma}$ of abelian groups such that $M_{\gamma}\cdot R_{\gamma'}\subseteq M_{\gamma+\gamma'}$ for all $\gamma,\gamma'\in\Gamma$.   

(iii) Given two (right) graded $R$-modules $M=\oplus_{\gamma\in\Gamma}M_{\gamma}$ and $N=\oplus_{\gamma\in\Gamma}N_{\gamma}$ over $R$, a \emph{graded module homomorphism} $f:M\rightarrow N$ is defined to be a module homomorphism $f\colon M\rightarrow N$ such that $f(M_{\gamma})\subseteq N_{\gamma}$ for all $\gamma\in\Gamma$. 

(iv) We denote by $\GrMod R$ the \emph{category of graded $R$-modules} with graded module homomorphisms. We recall that this is an abelian category with enough projective and injective objects. We denote by $\GrProj R$ the class of projective objects in $\GrMod R$, that we call \emph{graded projective modules} and we also denote by $\GrInj R$ the class of injective objects in $\GrMod R$, that we call \emph{graded injective modules}. 

(v) Given a graded module $M$, one may consider the underlying module without the grading (which we also denote by $M$). It is well-known that a graded module is graded projective if and only if it is projective as a module without the grading, see \cite[Section 2.2]{methods}. However this is not the case for graded injective modules: a graded module that is injective is graded injective but a graded injective module might not be injective (see, for example, \cite[Remark 2.3.3]{methods}). The injective dimension of graded injective modules has been studied in \cite{vandenbergh}.  

(vi) Just like in the ungraded case, we will be treating graded $R$-modules as objects of $\mathsf{D}(\GrMod R)$, by viewing them as complexes concentrated in degree 0.

The following is the analogous (and is proven verbatim) to \cite[Proposition 2.1]{rickard} for graded rings. 

\begin{prop}
\label{basic_properties_of_localizing}
Let $\mathcal{C}$ be a localizing subcategory of $\mathsf{D}(\GrMod R)$. 
\begin{itemize}
\item[(i)] If $0\to X\to Y\to Z\to 0$ is a short exact sequence of complexes and two of them are in $\mathcal{C}$, then so is the third. 

\item[(ii)] If $X$ is in $\mathcal{C}$, then so is $X[t]$ for every $t\in\mathbb{Z}$. 

\item[(iii)] If $X$ and $Y$ are quasi-isomorphic and one of them is in $\mathcal{C}$ then so is the other.
 
\item[(iv)] If $\{X_i,i\in I\}$ is a set of objects in $\mathcal{C}$, then $\oplus_{i\in I}X_i$ is in $\mathcal{C}$.
 
\item[(v)] If $X\oplus Y$ is in $\mathcal{C}$, then both $X$ and $Y$ are in $\mathcal{C}$. 

\item[(vi)] If $X^*$ is a bounded complex and $X^i$ is in $\mathcal{C}$ for every $i$, then $X^*$ is in $\mathcal{C}$.   

\item[(vii)] If $X_0\xrightarrow{a_0}X_1\xrightarrow{a_1}X_2\xrightarrow{a_2}\dots$ is a sequence of cochain maps of complexes where $X_i$ is in $\mathcal{C}$ for all $i$, then $\varinjlim X_i$ is in $\mathcal{C}$.
 
\item[(viii)] If $X^*$ is a bounded above complex where $X^i$ is in $\mathcal{C}$ for every $i$, then $X^*$ is in $\mathcal{C}$.
    
\end{itemize}
\end{prop}

Recall that the \emph{twist} of a graded module $M$ by $\gamma\in\Gamma$, denoted by $M(\gamma)$, is defined to be a graded module with underlying module $M$ and grading given by $M(\gamma)_{\gamma'}=M_{\gamma+\gamma'}$. Twisting by $\gamma$ defines a functor $(\gamma)\colon\GrMod R\rightarrow \GrMod R$ given by $M\mapsto M(\gamma)$. 
As in the ungraded case and from the above, we have the following result. 

\begin{cor} \label{twists_generate}
The graded projective $R$-modules generate $\mathsf{D}(\GrMod R)$. In particular,
the graded modules $\{R(\gamma),\gamma\in\Gamma\}$ generate $\mathsf{D}(\GrMod R)$. 
\end{cor}
\begin{proof}
The first part is proved as in \cite[Proposition~2.2]{rickard} using Proposition~\ref{basic_properties_of_localizing}. For the second claim,  every graded projective $R$-module is a direct summand of direct sum of copies of $R(\gamma)$ for various $\gamma\in \Gamma$. Thus, every graded projective module belongs in $\mathsf{Loc}(R(\gamma),\gamma\in\Gamma)$ and since the graded projective modules generate $\mathsf{D}(\GrMod R)$ the proof is complete. 
\end{proof}

\begin{lem} \label{twists} 
Let $S$ be a class of objects in $\GrMod R$ that is closed under twists. Assume that $X$ is a graded $R$-module with $X\in\mathsf{Loc}(S)$. Then $X(\gamma)\in\mathsf{Loc}(S)$ for every $\gamma\in\Gamma$. 
\end{lem}
\begin{proof}
Consider the twist functor $(\gamma)\colon \GrMod R\to \GrMod R$, which is exact. By assumption, it sends elements of $S$ to elements of $S$. Therefore, by Lemma \ref{preimage_is_localizing}, since $(\gamma)$ preserves coproducts, the preimage of $\mathsf{Loc}(S)$ under the induced triangle functor $(\gamma)\colon\mathsf{D}(\GrMod R)\rightarrow \mathsf{D}(\GrMod R)$ is a localizing subcategory that contains $S$ and thus it also contains $\mathsf{Loc}(S)$. Therefore, $X(\gamma)\in \mathsf{Loc}(S)$ for every graded $R$-module $X\in\mathsf{Loc}(S)$.  
\end{proof}

The following definition is the graded analogue to \cite[Definition 3.1]{rickard} and it is the main notion of study in this paper. 

\begin{defn}
\label{defn_gradedinjgen}
We say that \emph{graded injectives generate} for a graded ring $R$ if 
\begin{center}
	$\mathsf{Loc}(\GrInj R)=\mathsf{D}(\GrMod R)$
\end{center}
\end{defn}

\begin{note}
If graded injectives generate for a graded ring $R$, then graded injectives generate for every twist $R(\gamma)$, as the categories $\GrMod R$ and $\GrMod R(\gamma)$ are equivalent. 
\end{note}

For a commutative Noetherian ring $R$, Neeman proved in \cite{chromatic} that there is a bijection as follows 
\[
\{\text{subsets of }\mathsf{Spec}R\}\stackrel[]{}{\rightleftarrows}\{\text{localizing subcategories of }\mathsf{D}(R)\}
\]
Using this Rickard proved that injectives generate for every commutative Noetherian ring, see \cite[Theorem 3.3]{rickard}. In \cite{IvoStevenson}, Dell'Amrogio and Stevenson generalise the result of Neeman to the context of \emph{$\epsilon$-commutative graded Noetherian rings}, recalled below. Using their result and with similar arguments as the ones of Rickard, we will show that graded injectives generate for any such graded ring. On one hand this is an expected analogue to the aforementioned theorem of Rickard. On the other hand, by a simple comparison between graded injective generation and injective generation, this will in fact imply the ungraded commutative case.

\begin{defn} \textnormal{(\!\!\cite[Definition 2.4]{IvoStevenson})} Let $\Gamma$ be an abelian group, $R$ a graded ring over $\Gamma$ and $\epsilon\colon\Gamma\times \Gamma\rightarrow \mathbb{Z}/2\mathbb{Z}$ a symmetric $\mathbb{Z}$-bilinear map. We say that $R$ is \emph{$\epsilon$-commutative} if $r\cdot s=(-1)^{\epsilon(\mathsf{deg}r,\mathsf{deg}s)}s\cdot r$ for all homogeneous elements $r,s\in R$. 
\end{defn}

\begin{exmp}
If $R$ is commutative or graded-commutative, then it may be viewed as  $\epsilon$-commutative with $\epsilon=0$ or $\epsilon(\gamma_1,\gamma_2)=\gamma_1+\gamma_2$ respectively. 
\end{exmp}

Let us explain some notation that will follow. Consider a class $F$ of objects of $\GrMod R$. Then, we write $\mathsf{Loc}(F)_{\otimes}$ for the localizing subcategory of $\mathsf{D}(\GrMod R)$ that is generated by $\{X(\gamma) \ | \ \gamma\in\Gamma, X\in F\}$ and call it the \emph{localizing  $\otimes$-ideal generated by $F$}. There is a good reason the tensor sign $\otimes$ appears, but we deliberately avoid it, since it is beyond the scope of this text (in fact this is not the definition of a localizing tensor ideal, but from \cite[Lemma 2.21]{IvoStevenson}, it is equivalent to the one given here). The main theorem of \cite{IvoStevenson} states that for an $\epsilon$-commutative graded Noetherian ring $R$ there is a bijection as follows 

\[
\{\text{subsets of }\mathsf{Spec}^hR\}\stackrel[]{}{\rightleftarrows}\{\text{localizing tensor ideals of }\mathsf{D}(\GrMod R)\}
\]
As already advertised, the following is the graded-analogue to \cite[Theorem 3.3]{rickard}.

\begin{thm}
\label{injgenepsilongraded} 
Graded injectives generate for every $\epsilon$-commutative graded Noetherian ring $R$.  
\end{thm}
\begin{proof}
By \cite[Theorem 5.8]{IvoStevenson}, there is an inclusion preserving bijection between the localizing $\otimes$-ideals of $\mathsf{D}(\GrMod R)$ and the subsets of $\mathsf{Spec}^hR$. By \cite[Proposition 4.7]{IvoStevenson}, for any homogeneous prime ideal $p$ of $R$, the injective hull $I(R/p)$ belongs in $\mathsf{Loc}(k(p))_{\otimes}$ where $k(p)=R_p/pR_p$ and by \cite[Proposition 4.9]{IvoStevenson}, $\mathsf{Loc}(k(p))_{\otimes}$ is minimal, hence generated by $I(R/p)$ (as a localizing tensor ideal). Under the bijection mentioned above, $\mathsf{Loc}(k(p))_{\otimes}$ corresponds to $p$. Therefore the ideal $\mathsf{Loc}( \GrInj R)=\mathsf{Loc}(\GrInj R)_{\otimes}$, which contains all $I(R/p)$ (that is for every $p$), corresponds to $\mathsf{Spec}^hR$, which in turn corresponds to the whole derived category $\mathsf{D}(\GrMod R)$ and so $\mathsf{Loc}(\GrInj R)=\mathsf{D}(\GrMod R)$. 
\end{proof}

\begin{exmp}
\label{example_tensor}
Let $R$ be a commutative Noetherian ring and $M$ a $R$-$R$-bimodule that is nilpotent and finitely generated on both sides. Consider the tensor ring $T_R(M)$, that is $T_{R}(M)=R\oplus M\oplus M^{\otimes 2}\oplus M^{\otimes 3}\oplus\cdots$, with $M^{\otimes 0}=R$ and multiplication induced by $M^{\otimes k}\otimes M^{\otimes l}\rightarrow M^{\otimes k+l}$. Then, $T_R(M)$ is commutative Noetherian, thus by Theorem \ref{injgenepsilongraded} graded injectives generate for it. A non-commutative version of the latter is proved in Theorem \ref{mainthm2}.  
\end{exmp}
\begin{exmp}
\label{example_exterior}
Let $R$ be a commutative Noetherian ring and $M$ an $R$-$R$-bimodule that is nilpotent and finitely generated on both sides. Consider the exterior algebra $\bigwedge_R(M)$, that is $\bigwedge_R(M)=R\oplus \bigwedge^1_R(M)\oplus \bigwedge^2_R(M)\oplus\cdots$, with $\bigwedge^0M=R$ and multiplication induced by $(\bigwedge^k_R(M))\wedge (\bigwedge_R^l(M))\rightarrow \bigwedge_R^{k+l}(M)$. Then, $\bigwedge_R(M)$ is graded commutative Noetherian, thus by Theorem~\ref{injgenepsilongraded} graded injectives generate for it. 
\end{exmp}

We recall the following definition. 
\begin{defn}
Two graded rings $R$ and $S$ are called \emph{graded derived equivalent} if there exists a triangle equivalence between $\mathsf{D}(\GrMod R)$ and $\mathsf{D}(\GrMod S)$.
\end{defn}
As a last part of this section, we will show that graded injective generation is an invariant under graded derived equivalence among the rings that are graded over finite groups. To do this, we rely on the ungraded case, which we recall below.
\begin{thm} \textnormal{(\!\!\cite[Theorem 3.4]{rickard})}
\label{derived_invariance}
If $R$ and $S$ are derived equivalent rings, meaning that there is a triangle equivalence between $\mathsf{D}(R)$ and $\mathsf{D}(S)$, then injectives generate for $R$ if and only if injectives generate for $S$. 
\end{thm}

We also recall that for every graded ring $R$ that is graded over a finite abelian group, there exists the covering ring $\Hat{R}$ of $R$ such that the categories $\GrMod R$ and $\Mod \Hat{R}$ are equivalent, see \cite{smash, coverings}. An important consequence of the above is that injectives generate for $\Hat{R}$ if and only if graded injectives generate for $R$. The latter ring will be of central importance in Section 4 and we postpone its description until then. For now, just the existence of such a ring, together with Theorem \ref{derived_invariance} suffice for the following. 
\begin{cor}
Let $R$ and $S$ be two graded rings that are graded over finite groups. If $R$ and $S$ are graded derived equivalent, then graded injectives generate for $R$ if and only if graded injectives generate for $S$. 
\end{cor}
\begin{proof}
There is a ring $\Hat{R}$ such that $\GrMod R\simeq \Mod\hat{R}$ and there is a ring $\hat{S}$ such that $\GrMod S\simeq \Mod \hat{S}$. From the first equivalence it follows that if graded injectives generate for $R$, then injectives generate for $\hat{R}$. Since $\mathsf{D}(\GrMod R)\simeq \mathsf{D}(\GrMod S)$ it follows that there is a triangle equivalence $\mathsf{D}(\hat{R})\simeq \mathsf{D}(\hat{S})$. Therefore, by Theorem \ref{derived_invariance}, injectives generate for $\hat{S}$. Consequently, by the equivalence $\GrMod S\simeq \Mod \hat{S}$ we derive that graded injectives generate for $S$. 
\end{proof}

\subsection{Comparing the two conditions}
\label{comparingdef}
Here we establish a few elementary propositions that compare injective generation and graded injective generation. Recall that for any ring $R$ and any abelian group $\Gamma$, we may view $R$ as a graded ring over $\Gamma$ in a trivial way, by writing $R=\oplus_{\gamma\in\Gamma}R_{\gamma}$ with $R_0=R$ and $R_{\gamma}=0$ otherwise. In this case we say that $R$ is \emph{trivially graded}. Accordingly, any $R$-module $M$ can be viewed as a \emph{trivially graded module} over the trivially graded ring $R$ by writing $M=\oplus_{\gamma\in \Gamma} M_{\gamma}$ with $M_0=M$ and $M_{\gamma}=0$ otherwise. We begin with the following, that translates Definition \ref{injective_generation_defn} in our context. 

\begin{lem}
Injectives generate for a ring $R$ if and only if graded injectives generate for $R$ as a trivially graded ring over any abelian group $\Gamma$. 
\end{lem}
\begin{proof} 
View $R$ as trivially graded over any abelian group. Assume that injectives generate for $R$. The functor $\Mod R\rightarrow \GrMod R$ which maps a module $M$ to a trivially graded module maps injective modules to graded injective modules. Consequently, $R$ belongs in $\mathsf{Loc}(\GrInj R)$ and by Lemma \ref{twists} follows that $R(\gamma)\in\mathsf{Loc}(\GrInj R)$ for every $\gamma$, i.e. graded injectives generate for $R$. If now graded injectives generate for $R$, then the forgetful functor $\GrMod R\rightarrow \Mod R$ maps graded injective modules to injective modules (the universal property of injective modules is verified by viewing modules as trivially graded modules) and consequently $R\in\mathsf{Loc}(\Inj R)$, i.e. injectives generate for $R$. 
\end{proof}

Let $R=\oplus_{\gamma\in\Gamma}R_{\gamma}$ be a graded ring over an abelian group $\Gamma$. In order to compare the notions of injective generation and graded injective generation in greater generality we consider the following functors: 
\begin{center}
\begin{tikzcd}
	\Mod R \arrow[rr, "G"', bend right] &  & \GrMod R \arrow[ll, "F"', bend right]
\end{tikzcd}
\end{center}
 
The functor $F\colon \GrMod R\rightarrow\Mod R$ is the usual forgetful functor that maps a graded module to its underlying module (that is the same module but without the grading) and a graded module homomorphism to the underlying module homomorphism. The functor $G\colon\Mod R\rightarrow 
\GrMod R$ is defined as follows. A module $M$ over $R$ is mapped to $\oplus_{\gamma\in \Gamma}M_{\gamma}$ with $M_{\gamma}=M$ and $R$ acts on it in the following way: for $r_{\gamma}\in R_{\gamma}$ and $m_{\gamma'}\in M_{\gamma'}$, $m_{\gamma'}r_{\gamma}$ is the element of $M=M_{\gamma+\gamma'}$. Moreover, an $R$-module homomorphism $f\colon M\rightarrow N$ is mapped to $\oplus_{\gamma\in\Gamma}f$. It is evident that both $F$ and $G$ are exact. Moreover, the following holds.

\begin{prop} \textnormal{(\!\!\cite[Theorem 2.5.1]{methods})}
\label{adjoint}
The pair $(F,G)$ is an adjoint pair. If the abelian group $\Gamma$ is finite, then $(F,G,F)$ is an adjoint triple.
\end{prop}

We thus obtain the next result.
\begin{prop} 
\label{ungradedimpliesgraded}
Let $R$ be a graded ring over an abelian group $\Gamma$. If injectives generate for $R$, then graded injectives generate for $R$. 
\end{prop}
\begin{proof}
Consider the functor $G\colon\Mod R\to \GrMod R$, as defined before. By Proposition \ref{adjoint}, $G$ is right adjoint to $F$ which is exact, thus it maps injective modules to graded injective modules. Moreover $G$ is exact and preserves coproducts. Therefore, since injectives generate for $R$, by Proposition \ref{image_is_contained_in_localizing}, the image of $G\colon\mathsf{D}(\Mod R)\rightarrow\mathsf{D}(\GrMod R)$ is a subcategory of $\mathsf{Loc}( \GrInj R)$. We have $G(R)=\oplus_{\gamma\in\Gamma}R=R(\gamma)\oplus (\oplus_{\gamma'\in\Gamma\setminus\{\gamma\}}R(\gamma'))$ and since $\mathsf{Loc} (\GrInj R)$ is closed under summands, by Lemma \ref{twists}, $R(\gamma)$ is in $\mathsf{Loc} (\GrInj R)$ for every $\gamma\in\Gamma$, so by Lemma \ref{twists_generate} injectives generate for $R$ as a graded ring. 
\end{proof}

The inverse of the above holds in the case of finite gradings.  
\begin{prop}
\label{finitegrading} 
Let $R$ be a graded ring over an abelian group $\Gamma$. If $\Gamma$ is finite and graded injectives generate for $R$, then injectives generate for $R$.
\end{prop}
\begin{proof}
Consider the forgetful functor $F\colon\GrMod R\lxr \Mod R$. Since $\Gamma$ is assumed to be finite, it follows by Proposition \ref{adjoint} that $F$ is right adjoint to $G$. Since $G$ is exact, we derive that $F$ maps graded injective modules to injective modules. Moreover, $F$ is exact and preserves coproducts. Therefore, since graded injectives generate for $R$, it follows by Proposition \ref{image_is_contained_in_localizing} that the image of $F\colon\mathsf{D}(\GrMod R)\lxr \mathsf{D}(\Mod R)$ is a subcategory of $\mathsf{Loc} (\Inj R)$ and since $F(R)=R$, it follows that $R\in\mathsf{Loc}(\Inj R)$, thus by Remark \ref{generation_by_projectives} injectives generate for $R$. 
\end{proof}

We do not know an example of a graded ring over an infinite group for which graded injectives generate but injectives don't.
By combining the above together with Theorem~\ref{injgenepsilongraded}, we deduce the following.

\begin{cor}
\label{injgenepsilongradednoeth}
    Injectives generate for every $\epsilon$-commutative graded Noetherian ring that is graded over a finite abelian group. 
\end{cor}

We close this section with the next example.

\begin{exmp}
    Let $R$ be a commutative Noetherian ring and $M$ a nilpotent $R$-$R$-bimodule $M$ that is finitely generated on both sides. Let $k$ be such that $M^{\otimes k}=0$ and $n$ be such that $n\geq 2k$. Then $\bigwedge_R(M)$ can be viewed as the underlying ring of the graded ring $R'=\oplus_{\mathbb{Z}/n\mathbb{Z}}R'_i$ with $R'_i=\bigwedge_R^i(M)$. By Corollary~\ref{injgenepsilongradednoeth} injectives generate for $\bigwedge_R(M)$. 
\end{exmp}

\section{Morita context rings and ladders}
\label{moritacontextrings}

Given a graded ring $R$ over a finite group $\Gamma$, we consider its \emph{covering ring}, which is a ring $\Hat{R}$ such that $\GrMod R\simeq \Mod \Hat{R}$. Given the latter equivalence and the results of the previous section, it follows that injectives generate for $R$ if and only if injectives generate for $\Hat{R}$. In this section we study injective generation for the covering ring $\Hat{R}$ by viewing it as a Morita context ring.

\subsection{Morita context rings} 

In this subsection we recall the definition of a Morita context ring and the basic recollement structure needed in the sequel. 

\begin{defn}
\label{defn_of_Morita}
Let $A$ and $B$ be rings, $N$ an $A$-$B$-bimodule, $M$ a $B$-$A$-bimodule, $\phi\colon M\otimes_AN\lxr B$ a $B$-$B$-bimodule homomorphism and $\psi\colon N\otimes_BM\lxr A$ a $A$-$A$-bimodule homomorphism. We assume further that $\phi(m\otimes n)m'=m\psi(n\otimes m')$ and $n\phi(m\otimes n')=\psi(n\otimes m)n'$ for all $m,m'\in M$ and $n,n'\in N$. We define the \emph{Morita context ring} associated to the \emph{Morita context} $\mathcal{M}=(A,N,B,M,\phi,\psi)$ to be the ring 
\begin{center}
	$\Lambda_{(\phi,\psi)}(\mathcal{M})=\begin{pmatrix}
		A & N \\
		M & B
	\end{pmatrix}$
\end{center}
where the addition of elements is componentwise and multiplication is given by 
\begin{center}
	$\begin{pmatrix}
		a & n \\
		m & b 
	\end{pmatrix} 
 \cdot
	\begin{pmatrix}
		a' & n'\\
		m' & b' 
	\end{pmatrix}=
	\begin{pmatrix}
		aa'+\psi(n\otimes m') & an'+nb' \\
		ma'+bm' & bb'+\phi(m\otimes n') 
	\end{pmatrix}$
\end{center}
\end{defn}
We simply write $\Lambda_{(\phi,\psi)}$ for $\Lambda_{(\phi,\psi)}(\mathcal{M})$. We recall a useful description of (left) modules over a Morita context ring $\Lambda_{\phi,\psi}$ in terms of modules over $A$ and $B$. 

\begin{defn}
Define a category $\mathcal{M}(\Lambda)$ with objects tuples $(X,Y,f,g)$ where $X\in A\lMod $, $Y\in B\lMod$, $f\in \mathsf{Hom}_B(M\otimes_AX,Y)$ and $g\in \mathsf{Hom}_{A}(N\otimes_BY,X)$ such that the following diagrams commute
\begin{center}
\begin{tikzcd}
N\otimes_BM\otimes_AX \arrow[r, "N\otimes f"] \arrow[d, "\psi\otimes \mathsf{id}_X"'] & N\otimes_BY \arrow[d, "g"] & M\otimes_AN\otimes_BY \arrow[r, "M\otimes g"] \arrow[d, "\phi\otimes \mathsf{id}_Y"'] & M\otimes_AX \arrow[d, "f"] \\
A\otimes_AX \arrow[r, "\cong"]                                               & X                          & B\otimes_BY \arrow[r, "\cong"]                                               & Y                         
\end{tikzcd}
\end{center}
A morphism $(a,b)\colon (X,Y,f,g)\lxr (X',Y',f',g')$ is given by a pair $(a,b)$ where $a\colon X\lxr X'$ is a morphism in $ A\lMod $ and $b\colon Y\lxr Y'$ is a morphism in $B\lMod $ such that the following diagrams commute 
\begin{center}
    \begin{tikzcd}
M\otimes_AX \arrow[r, "f"] \arrow[d, "M\otimes a"'] & Y \arrow[d, "b"] & N\otimes_BY \arrow[d, "N\otimes b"'] \arrow[r, "g"] & X \arrow[d, "a"] \\
M\otimes_AX' \arrow[r, "f'"]                        & Y'               & N\otimes_BY' \arrow[r, "g'"]                        & X'              
\end{tikzcd}
\end{center}
\end{defn}

We then have the next result.

\begin{thm} \textnormal{(\!\!\cite{moritagreen})}
    The categories $\Lambda\lMod $ and $\mathcal{M}(\Lambda)$ are equivalent.
\end{thm}

From now on we identify modules over a Morita context ring $\Lambda_{(\phi,\psi)}$ with objects in $\mathcal{M}(\Lambda)$. Let $\Psi_X$ and $\Phi_Y$ denote the following compositions 
\begin{center}
   \begin{tikzcd}
N\otimes_BM\otimes_AX \arrow[r, "\psi\otimes \mathsf{id}_X"] \arrow[rr, "\Psi_X", bend left] & A\otimes_AX \arrow[r, "\simeq"] & X & \!\!\!\! \! \! \! \! \! \! \! \! M\otimes_AN\otimes_BY \arrow[r, "\phi\otimes \mathsf{id}_Y"] \arrow[rr, "\Phi_Y", bend left] & B\otimes_BY \arrow[r, "\simeq"] & Y
\end{tikzcd}
\end{center}
We define the following functors between the categories $A\lMod, B\lMod$ and $\Lambda_{(\phi,\psi)}\lMod$. 

\begin{itemize}
\item[(i)] The functor $\mathsf{T}_A\colon A\lMod\rightarrow \Lambda_{(\phi,\psi)}\lMod$ is defined on objects $X\in A\lMod$ by $\mathsf{T}_A(X)=(X,M\otimes_AX,\mathsf{id_{M\otimes X}},\Psi_X)$ and $\mathsf{T}_A(a)=(a,M\otimes a)$ for a morphism $a\colon X\rightarrow X'$ in $A\lMod$. 

\item[(ii)] The functor $\mathsf{U}_A\colon\Lambda\lMod\rightarrow A\lMod$ is defined by $\mathsf{U}_A(X,Y,f,g)=X$ for a tuple $(X,Y,f,g)$ in $\Lambda_{(\phi,\psi)}\lMod$ and $\mathsf{U}_A(a,b)=a$ for a morphism $(a,b)\colon(X,Y,f,g)\rightarrow (X',Y',f',g')$ in $\Lambda_{(\phi,\psi)}\lMod$.

\item[(iii)] The functor $\mathsf{T}_B\colon B\lMod\rightarrow \Lambda_{(\phi,\psi)}\lMod$ is defined on objects $Y\in B\lMod$ by $\mathsf{T}_B(Y)=(N\otimes_BY,Y,\Phi_Y,\mathsf{id}_{N\otimes Y})$ and $\mathsf{T}_B(b)=(N\otimes b,b)$ for a morphism $b\colon Y\rightarrow Y'$ in $B\lMod$.

\item[(iv)] The functor $\mathsf{U}_B\colon\Lambda\lMod\rightarrow B\lMod$ is defined by $\mathsf{U}_B(X,Y,f,g)=Y$ for a tuple $(X,Y,f,g)$ in $\Lambda_{(\phi,\psi)}\lMod$ and $\mathsf{U}_B(a,b)=b$ for a morphism $(a,b)\colon(X,Y,f,g)\rightarrow (X',Y',f',g')$ in $\Lambda_{(\phi,\psi)}\lMod$.
\end{itemize}

By \cite[Proposition~2.4]{morita} we have the following recollement diagrams: 
\begin{center}
\begin{tikzcd}
A\lMod/\mathsf{Im\psi} \arrow[rr, "\mathsf{inc}"] &  & {\Lambda_{(\phi,\psi)}\lMod} \arrow[rr, "\mathsf{U}_B"] \arrow[ll, bend right] \arrow[ll, bend left] &  & B\lMod \arrow[ll, "\mathsf{T}_B"', bend right] \arrow[ll, bend left]
\end{tikzcd}
\begin{tikzcd}
B\lMod/\mathsf{Im\phi} \arrow[rr, "\mathsf{inc}"] &  & {\Lambda_{(\phi,\psi)}\lMod} \arrow[rr, "\mathsf{U}_A"] \arrow[ll, bend right] \arrow[ll, bend left] &  & A\lMod \arrow[ll, "\mathsf{T}_A"', bend right] \arrow[ll, bend left]
\end{tikzcd}
\end{center}
In case that $\phi=\psi=0$, which is the case that we are mostly interested in this paper, we get the following recollements of module categories. 

\begin{center}
\begin{tikzcd}
A\lMod \arrow[rr, "\mathsf{Z}_A"] &  & {\Lambda_{(0,0)}\lMod} \arrow[rr, "\mathsf{U}_B"] \arrow[ll, bend right] \arrow[ll, bend left] &  & B\lMod \arrow[ll, "\mathsf{T}_B"', bend right] \arrow[ll, bend left]
\end{tikzcd}
\begin{tikzcd}
B\lMod \arrow[rr, "\mathsf{Z}_B"] &  & {\Lambda_{(0,0)}\lMod} \arrow[rr, "\mathsf{U}_A"] \arrow[ll, bend right] \arrow[ll, bend left] &  & A\lMod \arrow[ll, "\mathsf{T}_A"', bend right] \arrow[ll, bend left]
\end{tikzcd}
\end{center}
where $\mathsf{Z}_A$ and $\mathsf{Z}_B$ are defined as follows: 
\begin{itemize}
    \item[(v)] The functor $\mathsf{Z}_A\colon A\lMod\to \Lambda\lMod$ is defined by $\mathsf{Z}_A(X)=(X,0,0,0)$ for an $A$-module $X$ and $\mathsf{Z}_A(a)=(a,0)$ for a homomorphism $a\colon X\to X'$ of $A$-modules. 
    \item[(vi)] The functor $\mathsf{Z}_B\colon B\lMod\to \Lambda\lMod$ is defined by $\mathsf{Z}_B(X)=(0,Y,0,0)$ for an $B$-module $Y$ and $\mathsf{Z}_B(b)=(0,b)$ for a homomorphism $b\colon Y\to Y'$ of $B$-modules. 
\end{itemize}

For a detailed exposition of the above we refer the reader to \cite{morita}. We are also interested in the case where $M$ or $N$ is $0$. This means that we deal with  triangular matrix rings. In that case modules (left or right) over a triangular matrix ring can be expressed as triples $(X,Y,f)$ similarly as above. We state below a characterization of modules of finite projective dimension over triangular matrix rings, see \cite[Lemma~2.4]{smalo} and \cite{trivial}. We provide a proof for convenience of the reader, for which we recall that every left projective module over a triangular matrix ring $\big(\begin{smallmatrix}
  C & N\\
  0 & D
\end{smallmatrix}\big)$ is of the form $(P,0,0)\oplus (N\otimes_DQ,Q,1)$ for some left projective $C$-module $P$ and some left projective $D$-module $Q$.  

\begin{lem}
\label{projective_dimension_over_triangular_matrix_rings}
    Consider a triangular matrix ring $\Lambda=\big(\begin{smallmatrix}
  C & N\\
  0 & D
\end{smallmatrix}\big)$. The following hold.
\begin{itemize}
\item[(i)]  Assume that $\pd_CN<\infty$. Then for any left $\Lambda$-module $(X,Y,f)$ we have $\pd_{\Lambda}(X,Y,f)<\infty$ if and only if $\pd_CX<\infty$ and $\pd_DY<\infty$.

\item[(ii)]  Assume that $\pd N_D<\infty$. Then for any right $\Lambda$-module $(X,Y,f)$ we have $\pd_{\Lambda}(X,Y,f)<\infty$ if and only if $\pd X_C<\infty$ and $\pd Y_D<\infty$. 
\end{itemize}
\end{lem}
\begin{proof}
We prove (i) since case (ii) can be proved similarly. For a $\Lambda$-module $(X,Y,f)$ there is a short exact sequence as below 
    \begin{center}
        $0\rightarrow (X,0,0)\rightarrow (X,Y,f)\rightarrow (0,Y,0)\rightarrow 0$
    \end{center}
    If we show that $\pd_{\Lambda}(X,0,0)<\infty$ if and only if $\pd_CX<\infty$ and $\pd_{\Lambda}(0,Y,0)<\infty$ if and only if $\pd_{D}Y<\infty$, then by the above short exact sequence the proof will be complete. To begin with, it is easily seen that $\pd_{\Lambda}(X,0,0)=\pd_CX$.  We will now show that $\pd_{\Lambda}(0,Y,0)<\infty$ if and only if $\pd_DY<\infty$. Assume that $\pd_{\Lambda}(0,Y,0)<\infty$ and consider a projective resolution of $(0,Y,0)$ as below 
    \[
    0\rightarrow (P_n,0,0)\oplus (N\otimes_D Q_n, Q_n,1)\rightarrow \cdots\rightarrow (P_0,0,0)\oplus (N\otimes_DQ_0,Q_0,1)\rightarrow (0,Y,0)
    \]
    Further, consider the functor $\mathsf{U}_2\colon\Mod\Lambda\rightarrow \Mod D$ given by $(X,Y,f)\mapsto Y$. Then, applying $\mathsf{U}_2$ to the above resolution gives a resolution of $Y$ of finite length and each term is projective. Consequently, it follows that $\pd_DY<\infty$. We will show that $\pd_{D}Y<\infty$ implies $\pd_{\Lambda}(0,Y,0)<\infty$. We proceed by induction on the projective dimension of $Y$. If $\pd_DY=0$, this means that $Y$ is projective. There is a short exact sequence as below 
    \begin{center}
        $0\rightarrow (N\otimes_D Y,0,0)\rightarrow (N\otimes_D Y,Y,1)\rightarrow (0,Y,0)\rightarrow 0$
    \end{center}
    The middle term is a projective $\Lambda$-module and for the left term we have that $\pd_{C}N\otimes_D Y\leq \pd_CN<\infty$. Hence, $\pd_{\Lambda}(N\otimes_D Y,0,0)<\infty$. By the short exact sequence we conclude that $\pd_{\Lambda}(0,Y,0)<\infty$. Assume now that the claim holds for all $D$-modules of projective dimension $\leq n-1$ and assume moreover that $\pd_DY=n$. Then there is a short exact sequence of $D$-modules
    \begin{center}
        $0\rightarrow Y'\rightarrow P\rightarrow Y\rightarrow 0$
    \end{center}
    with $\pd_DY'=n-1$ and $P$ projective, which gives rise to a short exact sequence 
    \begin{center}
        $0\rightarrow (0,Y',0)\rightarrow (0,P,0)\rightarrow (0,Y,0)\rightarrow 0$
    \end{center}
    By the induction hypothesis, the middle and the left terms have finite projective dimension as left $\Lambda$-modules, thus so does $(0,Y,0)$. 
\end{proof}

\subsection{Injective generation for Morita context rings} 
In this subsection we present several useful lemmata regarding injective generation of Morita context rings. Some of them can be formulated in terms of ladders of recollements of abelian categories in the sense of \cite{ladder}. 

\begin{defn}  (\!\!\cite[Definition~1.1]{ladder}) 
    Let $\mathcal{B}$ and $\mathcal{C}$ be two abelian categories with an adjoint triple between them as below
    \begin{center}
        \begin{tikzcd}
\mathcal{B} \arrow[rr, "\mathsf{e}"] &  & \mathcal{C} \arrow[ll, "\mathsf{r}^0", bend left] \arrow[ll, "\mathsf{l}^0"', bend right]
\end{tikzcd}
\end{center}
A \emph{ladder} is a diagram of additive functors as follows:
\[ 
\begin{minipage}{0.4\textwidth}
\xymatrix@C=0.5cm{
 \ \ \ \ \ \ \ \ \ \ \ \ \ \ \ \ \ \ \ \ \  \ \ \  \ \  \ \ \ \ \ \ \ \ \ \ \ \ \vdots &&& \\
 &&&  &&& \\
 \mathcal{B} \ar[rrr]^{\mathsf{e}} \ar @/^3.0pc/[rrr]^{\mathsf{l}^1}   \ar @/_3.0pc/[rrr]_{\mathsf{r}^1}  &&& \mathcal{C} \ar @/_4.5pc/[lll]_{\mathsf{l}^2} \ar @/^4.5pc/[lll]^{\mathsf{r}^2}
\ar @/_1.5pc/[lll]_{\mathsf{l}^0} \ar
 @/^1.5pc/[lll]^{\mathsf{r}^0} \\
 &&& &&& \\
 \ \ \ \ \ \ \ \ \ \ \ \ \ \ \ \ \ \ \ \ \  \ \ \  \ \  \ \ \ \ \ \ \ \ \ \ \ \ \ \vdots &&& \\
 }
\end{minipage}
\]
such that $(\mathsf{r}^{i},\mathsf{r}^{i+1})$ is an adjoint pair and $(\mathsf{l}^{i+1},\mathsf{l}^{i})$ is an adjoint pair for all $i\geq 0$. We say that the ladder has \emph{$l$-height} $n$ if there is a tuple $(\mathsf{l}^{n-1},\dots,\mathsf{l}^2,\mathsf{l}^1,\mathsf{l}^0)$ of consecutive left-adjoints. The \emph{$r$-height} of a ladder is defined similarly. The \emph{height} of the ladder is defined to be the sum of the $r$-height and the $l$-height. We say that a recollement of abelian categories 
    \begin{center}
        \begin{tikzcd}
\mathcal{A} \arrow[rr, "\mathsf{i}"] &  & \mathcal{B} \arrow[rr, "\mathsf{e}"] \arrow[ll, "\mathsf{q}"', bend right] \arrow[ll, "\mathsf{p}", bend left] &  & \mathcal{C} \arrow[ll, "\mathsf{r}^0", bend left] \arrow[ll, "\mathsf{l}^0"', bend right]
\end{tikzcd}
    \end{center}
    \emph{admits a ladder} if there is a ladder for the adjoint triple $(\mathsf{l}^0,\mathsf{e},\mathsf{r}^0)$. 
\end{defn}

We also recall the definition of a ladder of recollements of triangulated categories. 

\begin{defn}
Let $\mathcal{T}'$, $\mathcal{T}$ and $\mathcal{T}''$ be triangulated categories. A \emph{ladder of recollements} between them is a diagram as below 
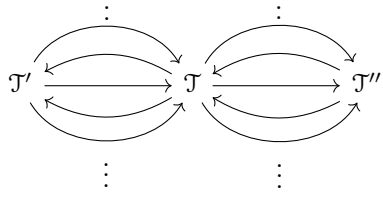
\begin{center}
    \begin{tikzcd}
\mathcal{T}' \arrow[rr] \arrow[rr, "\vdots", bend left=60] \arrow[rr, "\vdots"', bend right=60] &  & \mathcal{T} \arrow[rr] \arrow[ll, bend right] \arrow[ll, bend left] \arrow[rr, "\vdots", bend left=60] \arrow[rr, "\vdots"', bend right=60] &  & \mathcal{T}'' \arrow[ll, bend right] \arrow[ll, bend left]
\end{tikzcd}
\end{center}
such that any three consecutive rows form a recollement of triangulated categories. The number of recollements formed in a ladder of recollements of triangulated categories is called the \emph{height} of the ladder. 
\end{defn}

We want to lift recollements of abelian categories that admit a ladder to ladders of recollements of their derived categories, for which we need the following definition. 

\begin{defn} (\!\!\cite[Definition 3.6]{homological}) 
 An exact functor $\mathsf{i}\colon \mathcal{A}\to \mathcal{B}$ of abelian categories is called a \emph{homological embedding} if the induced map $\mathsf{i}^n_{X,Y}\colon \Ext^n_{\mathcal{A}}(X,Y)\rightarrow \Ext_{\mathcal{B}}^n(\mathsf{i}(X),\mathsf{i}(Y))$ is a group isomorphism for all $n\geq 0$ and all $X,Y\in\mathcal{A}$. 
\end{defn}

We then have the following result.

\begin{prop} \label{lifting_ladders}
    Let $(\Mod A,\Mod B,\Mod C)$ be a recollement with a ladder of height $n\geq 2$ such that the functor $\mathsf{i}\colon \Mod A\to \Mod B$ is a homological embedding. Then there is a recollement  $(\mathsf{D}(A),\mathsf{D}(B),\mathsf{D}(C))$ of derived module categories with  a ladder of height at least $n-1$. 
\end{prop}
\begin{proof}
    Let $(\Mod A,\Mod B,\Mod C)$ be a recollement of module categories as below 
    
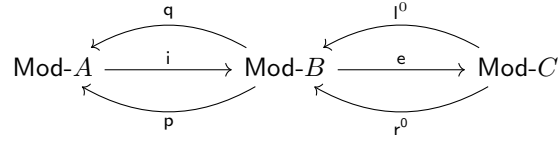
\begin{center}
        \begin{tikzcd}
\Mod A \arrow[rr, "\mathsf{i}"] &  & \Mod B \arrow[rr, "\mathsf{e}"] \arrow[ll, "\mathsf{q}"', bend right] \arrow[ll, "\mathsf{p}", bend left] &  & \Mod C \arrow[ll, "\mathsf{l}^0"', bend right] \arrow[ll, "\mathsf{r^0}", bend left]
\end{tikzcd}
    \end{center}
    Then, if we assume the above to be homological, it gives rise to a recollement of derived categories (see for example \cite{dalezios}): 
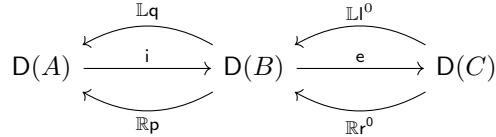
    \begin{center}
        \begin{tikzcd}
\mathsf{D}(A) \arrow[rr, "\mathsf{i}"] &  & \mathsf{D}(B) \arrow[rr, "\mathsf{e}"] \arrow[ll, "\mathbb{L}\mathsf{q}"', bend right] \arrow[ll, "\mathbb{R}\mathsf{p}", bend left] &  & \mathsf{D}(C) \arrow[ll, "\mathbb{L}\mathsf{l}^0"', bend right] \arrow[ll, "\mathbb{R}\mathsf{r^0}", bend left]
\end{tikzcd}
    \end{center}
    and by \cite[Proposition 3.2]{Angeleri}, it can be extended one step upwards if and only if $\mathbb{L}\mathsf{l^0}$ admits a left adjoint and it can be extended one step downwards if and only if $\mathbb{R}\mathsf{r}^0$ admits a right adjoint. Therefore, if the given recollement $(\Mod A,\Mod B,\Mod C)$ admits a ladder of $l$-height $m$ and $r$-height $n$ (in which case the height of the ladder is $m+n$), then it can be lifted to a recollement $(\mathsf{D}(A),\mathsf{D}(B),\mathsf{D}(C))$ which can be extended $m-1$ times upwards and $n-1$ times downwards, thus admits a ladder of height $m+n-1$. 
\end{proof}

We also need the following result of Cummings.
 
\begin{prop} \textnormal{(\!\!\cite[Proposition 6.10]{cummings})} \label{cummings_recollement}
    Let $(R)=(\mathsf{D}(A),\mathsf{D}(B),\mathsf{D}(C))$ be a recollement of derived categories of rings that belongs in a ladder of height 2. The following hold:
\begin{itemize} 

\item[(i)] If $(R)$ is on the bottom of the ladder and injectives generate for $B$, then injectives generate for $C$.

\item[(ii)] If $(R)$ is on top of the ladder and injectives generate for $B$, then injectives generate for $A$. 

\item[(iii)] If injectives generate for $A$ and $C$, then injectives generate for $B$. 
\end{itemize}
\end{prop}

By the above we obtain the next result.

\begin{prop}
\label{ladder}
Let $(\Mod A,\Mod B, \Mod C)$ be a recollement such that the functor $\mathsf{i}\colon \Mod A\to \Mod B$ is a homological embedding.
\begin{itemize} 
        \item[(i)] If the recollement admits a ladder with $l$-height at least 2 and injectives generate for $B$, then injectives generate for $C$.
        \item[(ii)] If the recollement admits a ladder with $r$-height at least 2 and injectives generate for $B$, then injectives generate for $A$. 
        \item[(iii)] If the recollement admits a ladder of height at least 3 and injectives generate for $A$ and $C$, then injectives generate for $B$. 
    \end{itemize}
\end{prop}
\begin{proof}
This follows by combining Proposition \ref{lifting_ladders} and Proposition \ref{cummings_recollement}. 
\end{proof}

We derive the following consequence for triangular matrix rings which is due to Cummings \cite[Example 6.11]{cummings}.

\begin{cor}
\label{injective_generation_for_triangular_matrix_rings}
Let $B=\big(\begin{smallmatrix}
  A & _AM_C\\
  0 & C
\end{smallmatrix}\big)$ be a triangular matrix ring. 
\begin{itemize}
\item[(i)] If injectives generate for $A$ and $C$ then injectives generate for $B$.
     
\item[(ii)] If injectives generate for $B$ then injectives generate for $C$. 
\end{itemize}
\end{cor}
\begin{proof}
It is well-known, see for instance \cite{recolsurvey}, that there is a recollement of module categories $(\Mod A,\Mod B,\Mod C)$ that admits a ladder with $l$-height at least 2 such that the middle left functor is a homological embedding. Then the result follows from Proposition~\ref{ladder}.
\end{proof}

\begin{cor}
    Let $\Lambda$ be a ring and $I$ any two-sided ideal of $\Lambda$. Consider the Morita context ring $\Gamma=\big(\begin{smallmatrix}
  \Lambda & I\\
  \Lambda & \Lambda
\end{smallmatrix}\big)$ with multiplication induced by the action of $\Lambda$ on $I$. If injectives generate for $\Lambda/I$ and $\Lambda$ then injectives generate for $\Gamma$.  
\end{cor}
\begin{proof}
By \cite{ladder} there is a recollement $(\Mod \big(\begin{smallmatrix}
  \Lambda/I & 0\\
  \Lambda/I & \Lambda/I
\end{smallmatrix}\big), \Mod \Gamma, \Mod \Lambda)$ that admits a ladder of height at least 4 such that the middle left functor is a homological embedding. Then the result follows from Proposition~\ref{ladder}.
\end{proof}

We end this section with the following theorem. 
  
\begin{thm}
\label{injectivegenerationformorita}
    Let $A$ and $B$ be two rings, $N$ a $A$-$B$-bimodule and $M$ a $B$-$A$-bimodule. Consider the Morita context ring $\Lambda=\big(\begin{smallmatrix}
  A & N\\
  M & B
\end{smallmatrix}\big)$  with $\phi=\psi=0$. 
\begin{itemize}
    \item[(i)]  Assume that $\pd_AN<\infty$. If injectives generate for $\Lambda$, then injectives generate for $A$. 
    \item[(ii)] Assume that $\pd_BM<\infty$. If injectives generate for $\Lambda$, then injectives generate for $B$. 
    \item[(iii)] Assume that $\pd_{\Lambda}\mathsf{Z}_A(A)<\infty$ and $\pd_{\Lambda}\mathsf{Z}_B(B)<\infty$. If injectives generate for $A$ and $B$, then injectives generate for $\Lambda$. 
\end{itemize}
\end{thm}
\begin{proof}
(i) Consider the following adjoint pair
\begin{center}
    \begin{tikzcd}
\mathsf{D}(\Lambda) \arrow[rr, "\mathsf{U}_A"] &  & \mathsf{D}(A) \arrow[ll, "\mathbb{L}\mathsf{T}_A"', bend right]
\end{tikzcd}
\end{center}
Note that we work with right modules and the functor $\mathsf{T}_A\colon\Mod A\rightarrow \Mod\Lambda$ is given by $X\mapsto (X,X\otimes_AN,1)$. By \cite[Lemma 6.1]{morita} (adjusted for right modules) we have $\mathsf{U}_B\mathbb{L}_n\mathsf{T}_A(-)\cong \mathsf{Tor}_n^A(-,N)$ and $\mathsf{U}_A\mathbb{L}_n\mathsf{T}_A=0$.  Since $\pd_AN<\infty$, it follows that for every $A$-module $X$ we have that $\mathbb{L}_n\mathsf{T}_A(X)=0$ for $n$ large enough. We conclude by Remark \ref{bounded_cohomology} that $\mathsf{U}_A$ maps bounded complexes of injective modules to bounded complexes of injective modules and since injectives generate for $\Lambda$ and $\mathsf{U}_A$ preserves coproducts, by Proposition \ref{image_is_contained_in_localizing}, we get that the image $\mathsf{Im}(\mathsf{U}_A)$ is a subcategory of $\mathsf{Loc}_A(\Inj A)$. In particular, $\mathsf{U}_A(A,0,0,0)=A$ is in $\mathsf{Loc}_A(\Inj A)$, so by Remark \ref{generation_by_projectives} we conclude that injectives generate for $A$. 
 
(ii) We work with the following adjoint pair
\begin{center}
   \begin{tikzcd}
\mathsf{D}(\Lambda) \arrow[rr, "\mathsf{U}_B"] &  & \mathsf{D}(B) \arrow[ll, "\mathbb{L}\mathsf{T}_B"', bend right]
\end{tikzcd}
\end{center} 
and with similar arguments as in (i), the result follows. 

(iii) Consider the ring homomorphisms $f\colon\Lambda\lxr A$ and $g\colon\Lambda\lxr B$ given by 
\begin{center}
    $\begin{pmatrix}
	a & n \\
	m & b
\end{pmatrix}\longmapsto a$ \  and  \ $\begin{pmatrix}
	a & n \\
	m & b
\end{pmatrix}\longmapsto b$
\end{center}
respectively. The assumption that $\pd_{\Lambda}\mathsf{Z}_A(A)<\infty$ is equivalent to $\pd_{\Lambda}A<\infty$ where the $\Lambda$-module structure on $A$ is induced by $f$. Likewise, the assumption that $\pd_{\Lambda}\mathsf{Z}_B(B)<\infty$ is equivalent to $\pd_{\Lambda}B<\infty$, where the $\Lambda$-module structure on $B$ is induced by $g$. Let $(-\otimes_{\Lambda}A,i_A)$ be the adjoint pair induced by $f$ and $(-\otimes_{\Lambda}B,i_B)$ be the adjoint pair induced by $g$. Consider the derived functors of the above, as below 
\begin{center}
    \begin{tikzcd}
\mathsf{D}(A) \arrow[rr, "i_A"] &  & \mathsf{D}(\Lambda) \arrow[ll, "-\otimes_{\Lambda}^{\mathsf{L}}A"', bend right] &  & \mathsf{D}(B) \arrow[rr, "i_B"] &  & \mathsf{D}(\Lambda) \arrow[ll, "-\otimes_{\Lambda}^{\mathsf{L}}B"', bend right]
\end{tikzcd}
\end{center}
Since $\pd_{\Lambda}A<\infty$, it follows by Remark \ref{bounded_cohomology} that $i_A$ maps bounded complexes of injectives to bounded complexes of injectives. Likewise, since $\pd_{\Lambda}B<\infty$, it follows that $i_B$ maps bounded complexes of injectives to bounded complexes of injectives. Both $i_A$ and $i_B$ preserve coproducts and since injectives generate for $A$ and $B$, we derive from Proposition \ref{image_is_contained_in_localizing} that $\mathsf{Im}(i_A)$ and $\mathsf{Im}(i_B)$ are subcategories of $\mathsf{Loc}_{\Lambda}(\Inj\Lambda)$. Consider the following short exact sequences of $\Lambda$-modules 
\begin{center} 
  $0\rightarrow\mathsf{Z}_B(M)\rightarrow\mathsf{T}_A(A)\rightarrow\mathsf{Z}_A(A)\rightarrow 0$
\end{center}
and 
\begin{center}
    $0\rightarrow\mathsf{Z}_A(N)\rightarrow \mathsf{T}_B(B)\rightarrow\mathsf{Z}_B(B)\rightarrow 0$
\end{center}
We notice that $\mathsf{Z}_B(M)$ is contained in the image of $i_B$ and $\mathsf{Z}_A(A)$ is contained in the image of $i_A$. Therefore, by the first short exact sequence, $\mathsf{T}_A(A)$ is contained in $\mathsf{Loc}_{\Lambda}(\Inj\Lambda)$. Likewise, $\mathsf{Z}_A(N)$ is contained in the image of $i_A$ and $\mathsf{Z}_B(B)$ is contained in the image of $i_B$. Therefore, by the second short exact sequence, $\mathsf{T}_B(B)$ is contained in $\mathsf{Loc}_{\Lambda}(\Inj\Lambda)$. Lastly, there is an isomorphism $\Lambda\cong \mathsf{T}_A(A)\oplus \mathsf{T}_B(B)$ of $\Lambda$-modules, thus $\Lambda\in \mathsf{Loc}_{\Lambda}(\Inj\Lambda)$ and so by Remark \ref{generation_by_projectives} we conclude that injectives generate for $\Lambda$. 
\end{proof}

\section{Covering rings vs Morita context rings} 
\label{section:covering}

In this section we study covering rings of graded rings via Morita context rings. 

\subsection{Covering rings}
We recall the notion of the covering ring of a graded ring. This is not something new; it is based on covering algebras in the sense of \cite{gabriel} and \cite{green} and it is closely related to smash products in the sense of \cite{smash}. For covering algebras we also refer the reader to \cite{derived} and \cite{coverings}.

\begin{defn}
\label{covering_ring_defn}
    Let $R=\oplus_{\gamma\in\Gamma}R_{\gamma}$ be a graded ring over a finite abelian group $\Gamma$. View $\Gamma$ as an indexed set and consider $\hat{R}$ to be the set of $\Gamma\times\Gamma$ matrices $(r_{gh})$ with $r_{gh}\in R_{h-g}$. By the next lemma, $\Hat{R}$ is a ring that we call the \emph{covering ring} of $R$. 
\end{defn}

\begin{lem}
\label{coveringiswelldefined}
    Let $R$ be a graded ring over a finite abelian group $\Gamma$ and consider the set $\Hat{R}$ as defined above. Then $\Hat{R}$ is closed under matrix multiplication and this endows it with the structure of a ring. 
\end{lem}
\begin{proof}
Once we show that $\Hat{R}$ is closed under matrix multiplication, all the other properties will follow trivially. Let $(a_{gh}), (b_{gh})\in\Hat{R}$. Consider the matrix multiplication $(c_{gh})=(a_{gh})(b_{gh})$. Then for the $(g,h)$ entry of $(c_{gh})$ we have $c_{gh}=\sum_{h'\in \Gamma} a_{gh'}b_{h'h}$. Moreover, $a_{gh'}\in R_{h'-g}$ and $b_{h'h}\in R_{h-h'}$. Therefore, $a_{gh'}b_{h'h}\in R_{h'-g}\cdot R_{h-h'}\subseteq R_{h-g}$. We conclude that the $(g,h)$ entry of $(c_{gh})$ belongs in $R_{h-g}$, implying that $(c_{gh})\in\Hat{R}$. 
\end{proof}

The subsequent result is known, we briefly recall its proof for readers convenience. 

\begin{prop} \textnormal{(\!\!\cite[Theorem 2.1]{smash},\!\!\!\cite[Theorem 2.5]{coverings})}
\label{equivalence}
    Let $R$ be a graded ring over a finite abelian group $\Gamma$ and consider its covering ring $\Hat{R}$. The categories $\GrMod R$ and $\Mod \Hat{R}$ are equivalent. 
\end{prop}
\begin{proof}
    Consider $U\colon\GrMod R\rightarrow \Mod\Hat{R}$ which maps a graded $R$-module $M=\oplus_{\gamma\in \Gamma}M_{\gamma}$ to the underlying group of $M$ and define the action by $\Hat{R}$ to be matrix multiplication by viewing elements of $M$ as row vectors $(m_{\gamma})$, i.e $(m_{\gamma})(r_{gh})=(m_{\gamma}')$ where $m_{\gamma}'=\sum_{h\in\Gamma}m_hr_{h \gamma}$. Moreover, a morphism of graded $R$-modules $f\colon\oplus_{\gamma\in\Gamma}M_{\gamma}\rightarrow \oplus_{\gamma\in\Gamma}N_{\gamma}$ is mapped to the underlying group homomorphism which is easily seen to respect the action defined above. We leave it to the reader to check that $f$ is a morphism of graded $R$-modules precisely when $U(f)$ is a morphism of $\Hat{R}$-modules. The latter amounts to the functor $U$ being fully faithful. Let $X$ be an $\Hat{R}$-module. Consider the graded $R$-module $M=\oplus_{\gamma\in\Gamma}M_{\gamma}$ with $M_{\gamma}=Xe_{\gamma}$, where $e_{\gamma}$ is the $\Gamma\times \Gamma$ matrix with $e_{\gamma \gamma}=1$ and $e_{gh}=0$ for all $(g,h)\neq (\gamma,\gamma)$. We leave it to the reader to check that $U(M)$ is isomorphic to $X$. The latter amounts to the functor $U$ being essentially surjective. By the above we conclude that $U$ is an equivalence of categories. 
\end{proof}

The situation so far is summed up to the following 
\begin{center}
\begin{tikzcd}
\Mod R \arrow[rr, "G"', bend right] &  & \GrMod R \arrow[ll, "F"', bend right] \arrow[rr, "U"', bend right] &  & \Mod \Hat{R} \arrow[ll, "V"', bend right]
\end{tikzcd}
\end{center}
where $F$ and $G$ are as in Section~\ref{section:injgengraded} (see subsection~\ref{comparingdef}), $U$ is the equivalence defined in Proposition~\ref{equivalence} and $V$ is agreed to be the inverse equivalence of $U$. 
\begin{cor}
\label{injective_generation_for_covering}
    Let $R$ be a graded ring over a finite abelian group $\Gamma$. Then injectives generate for $\hat{R}$ if and only if injectives generate for $R$. 
\end{cor}
\begin{proof}
    Since $\Gamma$ is assumed to be finite, injectives generate for $R$ if and only if graded injectives generate for $R$ by Proposition~\ref{ungradedimpliesgraded} and by Proposition \ref{finitegrading}. By the equivalence of Proposition \ref{equivalence}, it follows that graded injectives generate for $R$ if and only if injectives generate for $\Hat{R}$. 
\end{proof}

\subsection{The covering ring as a Morita context ring}
The key-ingredient to a lot of what follows is viewing the covering ring of a graded ring $R$ over a finite group $\Gamma$ as a Morita context ring. This can be done in general for any group $\Gamma$, but for our purposes and for the sake of simplicity, we present the following in the case $\Gamma=\mathbb{Z}/n \mathbb{Z}$. 

\begin{lem}
\label{coveringismorita}
Let $R$ be a graded ring over $\mathbb{Z}/n\mathbb{Z}$. The covering ring $\Hat{R}$ can be viewed as a Morita context ring. 
\end{lem}
\begin{proof}
      In this case the covering ring of $R$ is written below 
      \begin{center}
          $\Hat{R}=\begin{pmatrix}
		R_0 & R_1 & \cdots & R_{n-1} \\
		R_{n-1} & R_0 & \cdots & R_{n-2} \\ 
            \vdots & \vdots &  & \vdots \\ 
            R_1 & R_2 & \cdots & R_0
	\end{pmatrix}$
      \end{center}
      For any $k\in\{0,1,\dots,n-2\}$ we consider the following matrices: 
      \begin{center}
          $A=\begin{pmatrix}
		R_0 & R_1 & \cdots & R_{k} \\
		R_{n-1} & R_0 & \cdots & R_{k-1} \\ 
            \vdots & \vdots &  & \vdots \\ 
            R_{n-k} & R_{n-k+1} & \cdots & R_0
	\end{pmatrix}$
         $N=\begin{pmatrix}
		R_{k+1} & R_{k+2} & \cdots & R_{n-1} \\
		R_{k} & R_{k+1} & \cdots & R_{n-2} \\ 
            \vdots & \vdots &  & \vdots \\ 
            R_1 & R_2 & \cdots & R_{n-k-1}
	\end{pmatrix}$
      \end{center}
      \begin{center}
          $M=\begin{pmatrix}
		R_{n-k-1} & R_{n-k} & \cdots & R_{n-1} \\
		R_{n-k-2} & R_{n-k-1} & \cdots & R_{n-2} \\ 
            \vdots & \vdots &  & \vdots \\ 
            R_1 & R_2 & \cdots & R_{k+1}
	\end{pmatrix}$
         $B=\begin{pmatrix}
		R_0 & R_1 & \cdots & R_{n-k-2} \\
		R_{n-1} & R_0 & \cdots & R_{n-k-3} \\ 
            \vdots & \vdots &  & \vdots \\ 
            R_{k+2} & R_{k+3} & \cdots & R_0
	\end{pmatrix}$
      \end{center}
      Then an easy check shows that $A$ and $B$ are rings, $N$ is an $A$-$B$-bimodule and $M$ is a $B$-$A$-bimodule. Moreover, $\phi\colon M\otimes_AN\rightarrow B$ induced by matrix multiplication of an element of $M$ with an element of $N$ is a $B$-$B$-bimodule homomorphism. Similarly, define $\psi\colon N\otimes_BM\rightarrow A$ induced by matrix multiplicaion of an element of $N$ with an element of $M$ which is an $A$-$A$-bimodule homomorphism. The Morita context ring 
      $\big(\begin{smallmatrix}
  A & N\\
  M & B
\end{smallmatrix}\big)_{(\phi,\psi)}$ is isomorphic to $\Hat{R}$. 
\end{proof}

Let us look in detail the following example. 
\begin{exmp}
\label{basic_example}
    Let $R=\oplus_{n\in\mathbb{Z}/4\mathbb{Z}}R_n$ be a graded ring over $\mathbb{Z}/4\mathbb{Z}$. Consider the covering ring of $R$:
    \begin{center}
        $\Hat{R}=\begin{pmatrix}
		R_0 & R_1 & R_2 & R_3 \\
		R_3 & R_0 & R_1 & R_2 \\ 
            R_2 & R_3 & R_0 & R_1 \\ 
            R_1 & R_2 & R_3 & R_0
	\end{pmatrix}$
    \end{center}
    In view of Lemma \ref{coveringismorita}, we may view $\Hat{R}$ as a Morita context ring in the following 3 ways: 
    \begin{center}
$\begin{pmatrix}
  \begin{matrix}
  R_0
  \end{matrix}
  & \vline &  \begin{matrix}
  R_1 & R_2 & R_3
  \end{matrix}\\
    \hline
  \begin{matrix}
  R_3 \\ 
  R_2 \\
  R_1 
  \end{matrix} & \vline &
  \begin{matrix}
  R_0 & R_1 & R_2 \\
  R_3 & R_0 & R_1 \\ 
  R_2 & R_3 & R_0
  \end{matrix}
\end{pmatrix}
\begin{pmatrix}
  \begin{matrix}
  R_0 & R_1 \\ 
  R_3 & R_0 
  \end{matrix}
  & \vline &  \begin{matrix}
  R_2 & R_3 \\  
  R_1 & R_2 
  \end{matrix}\\
    \hline
  \begin{matrix}
  R_2 & R_3 \\ 
  R_1 & R_2 
  \end{matrix} & \vline &
  \begin{matrix}
  R_0 & R_1 \\ 
  R_3 & R_0
  \end{matrix}
\end{pmatrix}
\begin{pmatrix}

   \begin{matrix}
  R_0 & R_1 & R_2 \\
  R_3 & R_0 & R_1 \\ 
  R_2 & R_3 & R_0
  \end{matrix}
  & \vline &  \begin{matrix}
  R_3 \\
  R_2 \\ 
  R_1 
  \end{matrix}\\
    \hline
  \begin{matrix}
  R_1 & R_2 & R_3
  \end{matrix} & \vline &
  \begin{matrix}
  R_0 
  \end{matrix}
\end{pmatrix}
$
\end{center}
In any case, $A$ is the top left block, $M$ is the bottom left block, $N$ is the top right block and $B$ is the bottom right block. 
\end{exmp}

From now on, we will be particularly interested in graded rings over $\mathbb{Z}/2^n\mathbb{Z}$ with $R_i=0$ for $i\in \{2^{n-1}, \ldots, 2^n-1\}$. We will view its covering ring $\Hat{R}$ as a Morita context ring obtained by "splitting" it in half, i.e. $\Hat{R}= \big(\begin{smallmatrix}
  A & M\\
  M & A
\end{smallmatrix}\big)$ where $A$ and $M$ are given by $2^{n-1}\times 2^{n-1}$ matrices as below 
\[
   A=\begin{pmatrix}
		R_0 & R_{1} & \cdots & R_{2^{n-1}-1}\\
		R_{2^n-1}=0 & R_0 & \cdots & R_{2^{n-1}-2} \\ 
		\vdots & \vdots & & \vdots \\ 
		R_{2^{n-1}+1}=0 & R_{2^{n-1}+2}=0 & \cdots & R_0
	\end{pmatrix}
\]
and 
\[
        M=\begin{pmatrix}
		R_{2^{n-1}}=0 & R_{2^{n-1}+1}=0 & \cdots & R_{2^{n}-1}=0\\
		R_{2^{n-1}-1} & R_{2^{n-1}}=0 & \cdots & R_{2^{n}-2}=0 \\ 
		\vdots & \vdots & & \vdots \\ 
		R_{1} & R_{2} & \cdots & R_{2^{n-1}}=0
	\end{pmatrix}
\]
It is evident that $A$ is triangular. Moreover, the latter is a Morita context ring with zero bimodule homomorphisms, as shown in the following lemma. 

\begin{lem}
\label{lem:coveringzerozero}
Let $R=\oplus_{i\in\mathbb{Z}/2^n\mathbb{Z}}R_i$ be a graded ring over $\mathbb{Z}/2^n\mathbb{Z}$ for some $n\geq 1$. Assume that $R_i=0$ for all $i\in\{2^{n-1},\dots,2^n-1\}$. Then the covering ring $\Hat{R}$ is a Morita context ring with zero bimodule homomorphisms.
\end{lem}
\begin{proof}
    Consider the covering ring $\Hat{R}$ which we view as a Morita context ring $\big(\begin{smallmatrix}
  A & M\\
  M & A
\end{smallmatrix}\big)$. Let $M_a$ denote the $a$-th row of $M$, with $0\leq a\leq 2^{n-1}-1$ and $M_b$ denote the $b$-th column of $M$, with $0\leq b\leq 2^{n-1}-1$. Pictorially the situation is as below 
\[
M_a=\overbrace{\begin{pmatrix}
    R_{2^{n-1}-a} & R_{2^{n-1}-a+1} & \cdots & R_{2^{n-1}-1} & R_{2^{n-1}}=0 & \cdots & R_{2^n-a-1}=0
\end{pmatrix}}^{\text{$2^{n-1}$ entries}}
\]
\[
\left. 
M_b=\begin{pmatrix} R_{2^{n-1}+b}=0 \\ \vdots \\ R_{2^{n-1}}=0 \\ R_{2^{n-1}-1} \\ \vdots \\ R_{b+2} \\ R_{b+1} \end{pmatrix} 
\right\} {\text{$2^{n-1}$ entries}}
\]
Multiplying an element of $M_a$ with an element of $M_b$ gives an element of $R_{2^{n}-a+b}$ which is 0 since $2^{n}-a+b\geq 2^{n-1}$. 
\end{proof}

In the following proposition we collect some important properties of Morita context rings that arise from graded rings over $\mathbb{Z}/2^n\mathbb{Z}$ (compare statements (iii) and (iv) below with \cite[Corollary 6.6]{minamoto}).

\begin{prop}
\label{reduction}
Let $R=\oplus_{i\in\mathbb{Z}/2^n\mathbb{Z}}R_i$ be a graded ring over $\mathbb{Z}/2^n\mathbb{Z}$ such that $R_i=0$ for all $i\in\{2^{n-1},\dots,2^n-1\}$. Consider the covering ring $\Hat{R}$ with respect to the given grading and view it as a Morita context ring $\big(\begin{smallmatrix}
  A & M\\
  M & A
\end{smallmatrix}\big)_{(0,0)}$. The following statements hold: 
\begin{enumerate}
\item[\textnormal{(i)}] If injectives generate for $R_0$, then injectives generate for $A$.
\item[\textnormal{(ii)}] If injectives generate for $A$, then injectives generate for $R_0$.
\item[\textnormal{(iii)}] If $\pd_{R_0}R_i<\infty$ for all $i$, then $\pd_AM<\infty$. 
\item[\textnormal{(iv)}] If $\pd{R_i}_{R_0}<\infty$ for all $i$, then $\pd M_A<\infty$. 
\item[\textnormal{(v)}] If $R_i$ is a nilpotent $R_0$-bimodule for every $i$, then the $A$-bimodule $M$ is nilpotent. 
\end{enumerate}
\end{prop}
\begin{proof} 
(i) and (ii) follow from a successive application of Corollary \ref{injective_generation_for_triangular_matrix_rings}. 
 
(iii) We begin with a simple observation: let $S$ be any graded ring over $\mathbb{Z}/2^n\mathbb{Z}$. Then, the following is a graded ring over $\mathbb{Z}/2^{n-1}\mathbb{Z}$ 

\begin{center}
$\begin{pmatrix}   S_0 & S_1\\   S_{2^n-1} & S_0 \end{pmatrix}\oplus \begin{pmatrix}   S_2 & S_3\\   S_1 & S_2 \end{pmatrix}\oplus \begin{pmatrix}   S_4 & S_5\\   S_3 & S_4 \end{pmatrix}\oplus\cdots\oplus \begin{pmatrix}   S_{2^n-2} & S_{2^n-1}\\   S_{2^n-3} & S_{2^n-2} \end{pmatrix}$
\end{center}
Denote the above graded ring by $\Lambda$. If $\pd_{S_0}S_i<\infty$ for all $i$ and $S_{2^n-1}=0$, it follows by Lemma \ref{projective_dimension_over_triangular_matrix_rings} that $\pd_{\Lambda_0}\Lambda_i<\infty$ for all $i$. Starting with $R$, we may apply the above successivly and end up with $A\oplus M$, as shown below.

\begin{center}
\scalebox{1}{
\begin{tikzcd}
R_0\oplus R_1\oplus R_2\oplus \cdots \oplus R_{2^n-1} \arrow[d, "\text{induced }\mathbb{Z}/2^{n-1}\mathbb{Z}-\text{graded ring}"]                                                                                                                                                                                                                                                                                                                                                                                                                                                                                                                                \\
\begin{pmatrix}   R_0 & R_1\\   R_{2^n-1} & R_0 \end{pmatrix}\oplus \begin{pmatrix}   R_2 & R_3\\   R_1 & R_2 \end{pmatrix}\oplus \begin{pmatrix}   R_4 & R_5\\   R_3 & R_4 \end{pmatrix}\oplus\cdots\oplus \begin{pmatrix}   R_{2^n-2} & R_{2^n-1}\\   R_{2^n-3} & R_{2^n-2} \end{pmatrix} \arrow[d, "\text{induced }\mathbb{Z}/2^{n-2}\mathbb{Z}-\text{graded ring}"]                                                                                                                                                                                                                                                                                          \\
\begin{pmatrix}   R_0 & R_1 & R_2 & R_3\\   R_{2^n-1} & R_0 & R_2 & R_2 \\    R_{2^n-2} & R_{2^n-1} & R_0 & R_1 \\    R_{2^n-3} & R_{2^n-2} & R_{2^n-1} & R_0 \end{pmatrix}\oplus \cdots \oplus   \begin{pmatrix}   R_{2^n-4} & R_{2^n-3} & R_{2^n-2} & R_{2^n-1}\\   R_{2^n-5} & R_{2^n-4} & R_{2^n-3} & R_{2^n-2} \\    R_{2^n-6} & R_{2^n-5} & R_{2^n-4} & R_{2^n-3} \\    R_{2^n-7} & R_{2^n-6} & R_{2^n-5} & R_{2^n-4}  \end{pmatrix} \arrow[d, "\text{induced }\mathbb{Z}/2^{n-3}\mathbb{Z}-\text{graded ring}"] \\
\vdots \arrow[d, "\text{induced }\mathbb{Z}/2\mathbb{Z}-\text{graded ring}"]                                                                                                                                                                                                                                                                                                                                                                                                                                                                                                                                                                                     \\
\begin{pmatrix} 		R_0 & R_{1} & \cdots & R_{2^{n-1}-1}\\ 		R_{2^n-1} & R_0 & \cdots & R_{2^{n-1}-2} \\  		\vdots & \vdots & & \vdots \\  		R_{2^{n-1}+1} & R_{2^{n-1}+2} & \cdots & R_0 	\end{pmatrix}\oplus \begin{pmatrix} 		R_{2^{n-1}} & R_{2^{n-1}+1} & \cdots & R_{2^{n}-1}\\ 		R_{2^{n-1}-1} & R_{2^{n-1}} & \cdots & R_{2^{n}-2} \\  		\vdots & \vdots & & \vdots \\  		R_{1} & R_{2} & \cdots & R_{2^{n-1}} 	\end{pmatrix}                                                                                                                                                                                                                             
\end{tikzcd}}                                                               
\end{center}
Furthermore, by the assumption that $R_i=0$ for all $i\in\{2^{n-1},\cdots, 2^n-1\}$, it follows that at each step the graded components will have finite projective dimension over the initial subring. In particular, for the graded ring $A\oplus M$ appearing in the last step above, we infer that $\pd_AM<\infty$. (iv) This is similar to (iii). 

(v) Consider the ring homomorphism $R_0\rightarrow A$ given by $r\mapsto r\cdot 1_{2^{n-1}}$ where $1_{2^{n-1}}$ is the identity $(2^{n-1}\times2^{n-1})$ matrix. This gives rise to a surjective map $M\otimes_{R_0}M\rightarrow M\otimes_AM$ and inductively, a surjective map $M^{\otimes_{R_0}k}\rightarrow M^{\otimes_{A}k}$ for all $k\geq 2$. The $R_0$-bimodule $M$ is a direct sum of $R_i$'s. Therefore, by assumption, for $k$ large enough we have that $M^{\otimes_{R_0}k}=0$, thus also $M^{\otimes_{A}k}=0$. 
\end{proof}

\subsection{Strongly graded rings} 
We end this section with an application on strongly graded rings via the Morita context approach of coverings rings.

 \begin{defn}
    Let $R$ be a graded ring over an abelian group $\Gamma$. We say that $R$ is strongly graded if $R_{\gamma}\cdot R_{\gamma'}=R_{\gamma+\gamma'}$ for all $\gamma,\gamma'\in\Gamma$. 
\end{defn}
Assume now that $R$ is strongly graded over a finite group $\Gamma$ and consider the covering ring $\Hat{R}$. View $\Hat{R}$ as a Morita context ring $\big(\begin{smallmatrix}
  A & N\\
  M & B
\end{smallmatrix}\big)_{(\phi,\psi)}$ with $A=R_0$. 
We derive the following result which is due to Dade (compare it also with \cite[Theorem~2.12]{smash}).
 
\begin{prop} \textnormal{(\!\!\cite{dade})}
A graded ring $R$ over a finite group $\Gamma$ is strongly graded if and only if $\Mod \Hat{R}\simeq \Mod R_0$. 
\end{prop}
\begin{proof}
From the Morita context ring $\Hat{R}\cong \big(\begin{smallmatrix}
  A & N\\
  M & B
\end{smallmatrix}\big)_{(\phi,\psi)}$ with $A=R_0$, there is a recollement of module categories (see subsection~3.1):
    \begin{center}
       \begin{tikzcd}
\Mod B/\mathsf{Im\phi} \arrow[rr] &  & \Mod\Hat{R} \arrow[rr] \arrow[ll, bend right] \arrow[ll, bend left] &  & \Mod R_0 \arrow[ll, bend left] \arrow[ll, bend right]
\end{tikzcd}
\end{center}
Consequently, $\Mod\Hat{R}\simeq \Mod R_0$ if and only if $\phi$ is surjective. We observe that the latter is true precisely when $R$ is strongly graded. Indeed, the covering ring $\Hat{R}$ of $R$ is given as below 
\begin{center}
    $\Hat{R}=\begin{pmatrix}
  \begin{matrix}
  R_0
  \end{matrix}
  & \vline &  \begin{matrix}
  R_{\gamma_1} & R_{\gamma_2} & \cdots & R_{\gamma_n}
  \end{matrix}\\
    \hline
  \begin{matrix}
  R_{\gamma_n} \\ 
  R_{\gamma_{n-1}} \\
  \vdots \\
  R_{\gamma_1} 
  \end{matrix} & \vline &
  \begin{matrix}
  R_0 & R_{\gamma_1} & \cdots & R_{\gamma_{n-1}} \\
  R_{\gamma_n} & R_0 & \cdots & R_{\gamma_{n-2}} \\ 
  \vdots & \vdots &  & \vdots \\ 
  R_{\gamma_2} & R_{\gamma_3} & \cdots & R_0
  \end{matrix}
\end{pmatrix}$
\end{center}
for some indexing of the group $\Gamma$. The homomorphism $\phi\colon M\otimes_{R_0}N\rightarrow B$ is induced by the multiplication of the ring $R$ and, by above, it is surjective if and only if $R$ is strongly graded.
\end{proof}

By the above we immediately get the following result. 
 
\begin{cor} \label{strongly_graded}
Let $R$ be a strongly graded ring over a finite abelian group. Then injectives generate for $R$ if and only if injectives generate for $R_0$. 
\end{cor}

Note that the one implication of the above Corollary is \cite[Example 5.4]{cummings}.

\section{Perfect bimodules and tensor rings}
\label{section:perfectbimod}

In this section we prove Theorem~B as stated in the Introduction. This section is divided into three subsections. In the first one we recall the notion of perfect bimodules and study Morita context rings with such bimodules. In the second subsection we study tensor products of modules over Morita context rings and in the last subsection we prove the main result for tensor rings using also the covering theory developed in Section~\ref{section:covering}.

\subsection{Perfect Bimodules} 
We start with the notion of a perfect bimodule. This notion has been used by Chen and Lu \cite{perfect} in order to study Gorenstein projective modules for tensor algebras. Similar conditions have been studied before, see for instance \cite{trivial} and \cite{palmer}. For the purpose of studying injective generation on tensor rings, we actually need something less strong which we define below. 

\begin{defn}
\textnormal{(\!\!\cite[Definition 4.4]{perfect})}
\label{perfect}
Let $R$ be a ring. For a bimodule $_RM_R$ consider the following conditions: 
\begin{itemize}
\item[(i)] $\pd_RM<\infty$, 

\item[(ii)] $\pd M_R<\infty$, and
 
\item[(iii)] $\Tor_i^R(M,M^{\otimes j})=0$ for all $i,j\geq 1$.
\end{itemize}
If $M$ satisfies (i) and (iii), then it is called \emph{left perfect}. If $M$ is left perfect and satisfies (ii), then it is called \emph{perfect}.  
\end{defn}

In the following lemma we collect some useful properties of perfect bimodules.

\begin{lem}\textnormal{(\!\!\cite[Corollary 4.3, Lemma 4.5]{perfect})} 
\label{basic_properties_of_perfect} 
    Let $R$ be a ring and $M$ an $R$-bimodule. The following are equivalent: 
    \begin{itemize}
        \item[(i)] $\mathsf{Tor}_i^R(M,M^{\otimes j})=0$ for all $i,j\geq 1$. 
        \item[(ii)] $\mathsf{Tor}_i^R(M^{\otimes j},M)=0$ for all $i,j\geq 1$. 
        \item[(iii)] $\mathsf{Tor}_i^R(M^{\otimes s},M^{\otimes j})=0$ for all $i,s,j\geq 1$. 
    \end{itemize}
    Moreover, if any of the above conditions holds, then $\pd_RM^{\otimes i}\leq i\pd_RM$ and $\pd M^{\otimes i}_R\leq i\pd M_R$ for all $i\geq 1$. 
\end{lem} 

As a direct consequence we have the following. 

\begin{cor} \label{corollary_for_perfect}
    Let $R$ be a ring and $M$ an $R$-bimodule. If $_RM_R$ is left perfect, then $_RM^{\otimes i}_R$ is left perfect for all $i\geq 1$.
\end{cor}

In the next result we obtain some useful homological properties of Morita context rings under the presence of a perfect bimodule.

\begin{prop}
\label{lemmataki}
Let $\Lambda_{(0,0)}=\big(\begin{smallmatrix}
  A & M\\
  M & A
\end{smallmatrix}\big)$ be a Morita context ring with zero bimodule homomorphisms. Assume that $M$ is nilpotent and left perfect. Then we have $\pd_{\Lambda}(M,0,0,0)<\infty$ and $\pd_{\Lambda}(0,M,0,0)<\infty$. In particular, it follows that ${\pd_{\Lambda}(A,0,0,0)<\infty}$ and $\pd_{\Lambda}(0,A,0,0)<\infty$. 
\end{prop} 
\begin{proof}
Since $A=B$ we write $\mathsf{T}_1$ for $\mathsf{T}_A$ and $\mathsf{T}_2$ for $\mathsf{T}_B$. In order to prove that $\pd_{\Lambda}(M,0,0,0)<\infty$ we consider the following short exact sequences
    \begin{center}
        $0\rightarrow (0,M^{\otimes 2},0,0)\rightarrow (M,M^{\otimes 2},1,0)\rightarrow (M,0,0,0)\rightarrow 0$ 
    \end{center}
    \begin{center}
        $0\rightarrow (M^{\otimes 3},0,0,0)\rightarrow (M^{\otimes 3},M^{\otimes 2},0,1)\rightarrow (0,M^{\otimes 2},0,0)\rightarrow 0$
    \end{center}
    \begin{center}
        $\vdots$
    \end{center}
Similarly, in order to show that $\pd_{\Lambda}(0,M,0,0)<\infty$ we consider the following short exact sequences 
     \begin{center}
        $0\rightarrow (M^{\otimes 2},0,0,0)\rightarrow (M^{\otimes 2},M,0,1)\rightarrow (0,M,0,0)\rightarrow 0$ 
    \end{center}
    \begin{center}
        $0\rightarrow (0,M^{\otimes 3},0,0)\rightarrow (M^{\otimes 2},M^{\otimes 3},1,0)\rightarrow (M^{\otimes 2},0,0,0)\rightarrow 0$
    \end{center}
    \begin{center}
        $\vdots$
    \end{center}
In both cases and since $M$ is nilpotent, it suffices to prove that each of the middle terms in the above short exact sequences have finite projective dimension. Thus our aim is to show that $\pd_{\Lambda}\mathsf{T}_j(M^{\otimes i})<\infty$ for $j=1,2$ and for all $i\geq 1$. By Lemma \ref{basic_properties_of_perfect} we have $\pd_AM^{\otimes i}\leq i\pd_AM$ for all $i$ and therefore each $M^{\otimes i}$ has finite projective dimension. Consider a projective resolution of $M^{\otimes i}$ as below:
\[
0\rightarrow P_k\rightarrow \dots\rightarrow P_0\rightarrow M^{\otimes i}\rightarrow 0 
\]
By \cite[Lemma 6.1 (i),(ii)]{morita} we have that $\mathbb{L}_n\mathsf{T}_1(M^{\otimes i})\cong (0,\Tor_n^A(M,M^{\otimes i}),0,0)$ and $\mathbb{L}_n\mathsf{T}_2(M^{\otimes i})\cong (\Tor_n^A(M,M^{\otimes i}),0,0,0)$. Therefore, since $M$ is left perfect, $\mathbb{L}_n\mathsf{T}_1(M^{\otimes i})=0$ and $\mathbb{L}_n\mathsf{T}_2(M^{\otimes i})=0$ for all $n,i\geq 1$. Lastly, the functors $\mathsf{T}_1$ and $\mathsf{T}_2$ send projective modules to projective modules. Therefore, if we apply $\mathsf{T}_1$ and $\mathsf{T}_2$ to the above projective resolution of $M^{\otimes i}$, we get a finite projective resolution of $\mathsf{T}_1(M^{\otimes i})$ and $\mathsf{T}_2(M^{\otimes i})$ respectively. This completes the first part of the proof. For the second, we consider the following short exact sequences 
    \begin{center}
        $0\rightarrow (0,M,0,0)\rightarrow \mathsf{T}_1(A)\rightarrow (A,0,0,0)\rightarrow 0$
    \end{center}
    and 
    \begin{center}
        $0\rightarrow (M,0,0,0)\rightarrow \mathsf{T}_2(A)\rightarrow (0,A,0,0)\rightarrow 0$.
    \end{center}
Since the tuples $\mathsf{T}_1(A)$ and $\mathsf{T}_2(A)$ are projective $\Lambda_{(0,0)}$-modules,  we infer that $\pd_{\Lambda}(0,M,0,0)<\infty$ if and only if $\pd_{\Lambda}(A,0,0,0)<\infty$ and ${\pd_{\Lambda}(M,0,0,0)<\infty}$ if and only if $\pd_{\Lambda}(0,A,0,0)<\infty$. This concludes the proof. 
\end{proof}

\subsection{Tensor products.}
Let $\Lambda=\big(\begin{smallmatrix}
  A & N\\
  M & B
\end{smallmatrix}\big)$ be a Morita context ring. As explained in Section~\ref{moritacontextrings}, a right $\Lambda$-module can be written as a tuple $(X,Y,f,g)$ and a left $\Lambda$-module can be written as a tuple $(X',Y',f',g')$. In this case, we denote by $(X,Y,f,g)\otimes_{\Lambda}(X',Y',f',g')$ the  ordinary tensor product of the two (as a right and left modules respectively). We begin with the following result, whose proof can be found in \cite[Proposition 3.6.1]{kryov_tuganbaev}. 

\begin{prop} \label{computation_of_tensor_1}
    Let $\Lambda=\big(\begin{smallmatrix}
  A & N\\
  M & B
\end{smallmatrix}\big)$ be a Morita context ring, $(X,Y,f,g)$ a right $\Lambda$-module and $(X',Y',f',g')$ a left $\Lambda$-module. There is a functorial group isomorphism
\begin{center}
    $(X,Y,f,g)\otimes_{\Lambda} (X',Y',f',g')\cong (X\otimes_AX')\oplus (Y\otimes_BY')/H$
\end{center}
where $H$ is generated by elements of the form
\begin{center}
    $(g(y\otimes m)\otimes x'-x\otimes g'(n\otimes y'),f(x\otimes n)\otimes y'-y\otimes f'(m\otimes x'))$
\end{center}
for $x\in X,x'\in X',y\in Y,y'\in Y',n\in N$ and $m\in M$. 
\end{prop}

With the help of Proposition \ref{computation_of_tensor_1} we prove several useful lemmata. 

\begin{lem} \label{computation_for_tensor_2}
    Let $A$ be a ring, $N$ an $A$-bimodule and consider the triangular matrix ring $\Lambda=\big(\begin{smallmatrix}
  A & N\\
  0 & A
\end{smallmatrix}\big)$. If $(X,Y,f)$ is a right $\Lambda$-module and $Z$ is a left $A$-module, then there is a functorial group isomorphism: 
\begin{center}
    $(X,Y,f)\otimes_{\Lambda}(N\otimes_A Z,Z,1)\cong Y\otimes_AZ$
\end{center}
\end{lem}
\begin{proof}
By Proposition \ref{computation_of_tensor_1}, there is a group isomorphism
\[
(X,Y,f)\otimes_{\Lambda} (N\otimes_A Z,Z,1)\cong (X\otimes_A N\otimes_A Z)\oplus (Y\otimes_A Z)/H
\]
where $H$ is generated by elements of the form $(-x\otimes n\otimes z,f(x\otimes n)\otimes z)$. Define a group homomorphism 
\[
(X\otimes_AN\otimes_AZ)\oplus (Y\otimes_A Z)\rightarrow Y\otimes_A Z, \, (x\otimes n\otimes z,y\otimes z)\mapsto f(x\otimes n)\otimes z+y\otimes z 
\]
which is surjective with kernel $H$. 
\end{proof}

Let $\Lambda=\big(\begin{smallmatrix}
  A & N\\
  0 & A
\end{smallmatrix}\big)$ be as in the above lemma and consider the group $M=\big(\begin{smallmatrix}
  N^{\otimes 2} & N^{\otimes 3}\\
  N & N^{\otimes 2}
\end{smallmatrix}\big)$. The natural homomorphisms $N^{\otimes k}\otimes_A N^{\otimes l}\rightarrow N^{\otimes k+l}$ induce an action (left and right) of $\Lambda$ on $M$ which turn it into a $\Lambda$-bimodule. In the following lemmata we stick with the above notation.

\begin{lem} \label{computation_of_tensor}
    There is an isomorphism  
\begin{center}
    $_{\Lambda}M^{\otimes i}_{\Lambda}\cong \big(\begin{smallmatrix}
  N^{\otimes 2i} & N^{\otimes 2i+1}\\
  N^{\otimes 2i-1} & N^{\otimes 2i}
\end{smallmatrix}\big)$
\end{center}
of $\Lambda$-bimodules. 
\end{lem}
\begin{proof}
    We claim that for every $k,l\geq 1$ there is an isomorphism of $\Lambda$-$\Lambda$-bimodules as below 
\begin{align}
    \big(\begin{smallmatrix}
  N^{\otimes k} & N^{\otimes k+1}\\
  N^{\otimes k-1} & N^{\otimes k}
\end{smallmatrix}\big)\otimes_{\Lambda}\big(\begin{smallmatrix}
  N^{\otimes l} & N^{\otimes l+1}\\
  N^{\otimes l-1} & N^{\otimes l}
\end{smallmatrix}\big)\cong \big(\begin{smallmatrix}
  N^{\otimes k+l} & N^{\otimes k+l+1}\\
  N^{\otimes k+l-1} & N^{\otimes k+l}
\end{smallmatrix}\big)
\end{align}
Then, the result will follow inductively. Indeed, $\big(\begin{smallmatrix}
  N^{\otimes k} & N^{\otimes k+1}\\
  N^{\otimes k-1} & N^{\otimes k}
\end{smallmatrix}\big)$ as a right $\Lambda$-module is written as $(N^{\otimes k}\oplus N^{\otimes k-1},N^{\otimes k+1}\oplus N^{\otimes k},1)$ and $\big(\begin{smallmatrix}
  N^{\otimes l} & N^{\otimes l+1}\\
  N^{\otimes l-1} & N^{\otimes l}
\end{smallmatrix}\big)$ as a left $\Lambda$-module is written as $(N^{\otimes l}\oplus N^{\otimes l+1},N^{\otimes l-1}\oplus N^{\otimes l},1)$. Therefore, the group isomorphism of Lemma \ref{computation_for_tensor_2} is given as follows
\begin{center}
    $\big(\begin{smallmatrix}
  n_1 & n_2\\
  n_3 & n_4
\end{smallmatrix}\big)\otimes \big(\begin{smallmatrix}
  n_1' & n_2'\\
  n_3' & n_4'
\end{smallmatrix}\big)\mapsto \big(\begin{smallmatrix}
  n_2\otimes n_3' & n_2\otimes n_4'\\
  n_4\otimes n_3' & n_4\otimes n_4'
\end{smallmatrix}\big)$
\end{center}
which indeed shows that there is an isomorphism of groups as (5.1). It is easy to check that it respects the left and right action by $\Lambda$, which completes the proof. 
\end{proof}

In the spirit of \cite[Lemma~6.1]{morita}, see also \cite[Lemma~3.5]{Mao}, we prove the following.

\begin{lem} \label{Tor_1}
    Let $Z$ be a left $A$-module and assume that $\mathsf{Tor}_i^A(N,Z)=0$ for all $i\geq 1$. Moreover, let $(X,Y,f)$ be any right $\Lambda$-module. Then 
    \begin{center}
        $\mathsf{Tor}_n^{\Lambda}((X,Y,f),(N\otimes_A Z,Z,1))\cong \mathsf{Tor}_n^A(Y,Z)$
    \end{center}
    for all $n\geq 1$.
\end{lem}
\begin{proof}
    Consider a projective resolution of $Z$ 
    \begin{align}
        \cdots \rightarrow P_n\rightarrow \cdots\rightarrow P_0\rightarrow Z\rightarrow 0
    \end{align}
    Applying $Y\otimes_A-$ to (5.2) and deleting the first term, gives the following complex 
    \begin{align}
        \cdots\rightarrow Y\otimes_A P_n\rightarrow \cdots\rightarrow Y\otimes_A P_0
    \end{align}
    whose n-th homology computes $\mathsf{Tor}_n^A(Y,Z)$. On the other hand, by the assumption that $\mathsf{Tor}_i^A(N,Z)=0$ for all $i\geq 1$, it follows by \cite[Lemma 6.1]{morita} that applying the functor $\mathsf{T}_2$ to (5.2) gives a projective resolution of $\mathsf{T}_2(Z)$ as below, where $\mathsf{T}_2\colon\Mod A\rightarrow \Mod\Lambda$ denotes the functor given by $Y\mapsto (N\otimes_AY,Y,1)$ on objects. 
 \begin{align}
        \cdots\rightarrow \mathsf{T}_2(P_n)\rightarrow \cdots\rightarrow \mathsf{T}_2(P_0)\rightarrow \mathsf{T}_2(Z)\rightarrow 0
    \end{align}
    Therefore, if we apply $(X,Y,f)\otimes_{\Lambda}-$ to (5.4) and delete the first term, we get the following complex 
    \begin{align}
        \cdots\rightarrow (X,Y,f)\otimes_{\Lambda} \mathsf{T}_2(P_n)\rightarrow\cdots\rightarrow (X,Y,f)\otimes_{\Lambda} \mathsf{T}_2(P_0)
    \end{align}
    whose n-th homology computes $\mathsf{Tor}_n^{\Lambda}((X,Y,f),\mathsf{T}_2(Z))$. Consider the following commutative diagram
    \[
\begin{tikzcd}
\cdots \arrow[r] & Y\otimes_A P_n \arrow[rr, "Y\otimes f_n"]                                                                                    &  & Y\otimes_A P_{n-1} \arrow[r]                                                    & \cdots \\
\cdots \arrow[r] & {(X,Y,f)\otimes_{\Lambda}\!\mathsf{T}_2(P_n)} \arrow[rr, "{(X,Y,f)\otimes_{\Lambda}\!\mathsf{T}_2(P_n)}"] \arrow[u, "\cong"] &  & {(X,Y,f)\otimes_{\Lambda}\!\mathsf{T}_2(P_{n-1})} \arrow[u, "\cong "] \arrow[r] & \cdots
\end{tikzcd}
    \]
    where the vertical isomorphisms are given by Lemma \ref{computation_for_tensor_2}. We infer that the complexes (5.3) and (5.5) are isomorphic and the result follows. 
\end{proof}

Our last technical lemma is an application of the two latter.

\begin{lem}
\label{lifting_perfect}
Let $N$ be an $A$-bimodule and consider the $\Lambda$-bimodule $M=\big(\begin{smallmatrix}
  N^{\otimes 2} & N^{\otimes 3}\\
  N & N^{\otimes 2}
\end{smallmatrix}\big)$, where $\Lambda=\big(\begin{smallmatrix}
  A & N\\
  0 & A
\end{smallmatrix}\big)$. Assume that $\Tor_j^A(N^{\otimes s},N^{\otimes i})= 0$ for all $j,s,i\geq 1$. Then $\Tor_j^{\Lambda}(M^{\otimes s},M^{\otimes i})= 0$ for all $j,s,i\geq 1$.
\end{lem}
\begin{proof}
By Lemma \ref{computation_of_tensor}, it follows that $M^{\otimes s}_{\Lambda}$ as a tuple is written as 
\[
(N^{\otimes 2s}\oplus N^{\otimes 2s-1},N^{\otimes 2s+1}\oplus N^{\otimes 2s},1)
\]
and $_{\Lambda}M^{\otimes i}$ as a tuple is written as 
\[
(N^{\otimes 2i}\oplus N^{\otimes 2i+1},N^{\otimes 2i-1}\oplus N^{\otimes 2i},1)
\]
Therefore, it follows by Lemma~\ref{Tor_1} that 
\[
{\Tor}_j^{\Lambda}(M^{\otimes s},M^{\otimes i})\cong {\Tor}_j^{A}(N^{\otimes 2s+1}\oplus N^{\otimes 2s}, N^{\otimes 2i-1}\oplus N^{\otimes 2i})  = 0
\]
and this concludes the proof.
\end{proof}

\subsection{Injective generation of tensor rings}

In this subsection, we apply the theory developed before to study injective generation of tensor rings with (left) perfect bimodules. This relies on the next result.

\begin{prop}
\label{lifting_perfect_2}
    Let $R$ be a ring and $M$ a nilpotent $R$-$R$-bimodule. Assume that $M^{\otimes k}=0$ and let $n$ be such that $2^{n-1}\geq k$. Consider the graded ring $R'=\oplus_{i\in\mathbb{Z}/2^n\mathbb{Z}}R_i'$ with $R_i'=M^{\otimes i}$ where $M^{\otimes 0}=R$ and the Morita context ring $\big(\begin{smallmatrix}
  A & M'\\
  M' & A
\end{smallmatrix}\big)$ associated to $R'$. If $M$ is left perfect as an $R$-$R$-bimodule, then $M'$ is left perfect as an $A$-$A$-bimodule.  
\end{prop}
\begin{proof}
Since $M$ is left perfect, it follows from Lemma \ref{basic_properties_of_perfect} that $\pd_RM^{\otimes i}<\infty$ for all $i$. Therefore, it follows from Proposition \ref{reduction} that $\pd_AM'<\infty$. Thus we have to show that $\Tor^A_j(M',M'^{\otimes i})= 0$ for all $i,j\geq 1$. For the purposes of this proof, we introduce the following notation. We say that a graded ring $R_0\oplus R_1\oplus R_2\oplus \cdots$ satisfies the Tor-vanishing condition if 
\[
\mathsf{Tor}_i^{R_0}(R_k^{\oplus s},R_l^{\otimes j})=0
\]
for all $i,s,j,k,l\geq 1$. By the assumption, Lemma \ref{computation_of_tensor} and Lemma \ref{lifting_perfect}, the following graded ring (which is $\mathbb{Z}/2^{n-1}\mathbb{Z}$-graded) satisfies the Tor-vanishing condition. 
\[
\begin{pmatrix}     R & M \\      0 & R  \end{pmatrix}\oplus \begin{pmatrix}     M^{\otimes 2} & M^{\otimes 3} \\      M & M^{\otimes 2}  \end{pmatrix} \oplus \begin{pmatrix}
    M^{\otimes 4} & M^{\otimes 5} \\ 
    M^{\otimes 3} & M^{\otimes 4}
\end{pmatrix} \oplus  \begin{pmatrix}
    M^{\otimes 6} & M^{\otimes 7} \\ 
    M^{\otimes 5} & M^{\otimes 6}
\end{pmatrix} \oplus  \cdots
\]
Then, by Lemma~\ref{computation_of_tensor} and Lemma~\ref{lifting_perfect}, we apply the above observation successively as shown below 

\begin{center}
\scalebox{0.95}{
\begin{tikzcd}
\begin{pmatrix}     R & M \\      0 & R  \end{pmatrix}\oplus \begin{pmatrix}     M^{\otimes 2} & M^{\otimes 3} \\      M & M^{\otimes 2}  \end{pmatrix} \oplus \begin{pmatrix}
    M^{\otimes 4} & M^{\otimes 5} \\ 
    M^{\otimes 3} & M^{\otimes 4}
\end{pmatrix} \oplus  \begin{pmatrix}
    M^{\otimes 6} & M^{\otimes 7} \\ 
    M^{\otimes 5} & M^{\otimes 6}
\end{pmatrix} \oplus  \cdots  \arrow[d, "\text{induced }\mathbb{Z}/2^{n-2}\mathbb{Z}-\text{graded ring}"]                                                                                                                                                                                                                                                                                                                                                                                                                                                                                                                                                        \\
\begin{pmatrix}     R & M & M^{\otimes 2} & M^{\otimes 3} \\      0 & R & M & M^{\otimes 2} \\      0 & 0 & R & M \\      0 & 0 & 0 & R \end{pmatrix}\oplus \begin{pmatrix}     M^{\otimes 4} & M^{\otimes 5} & M^{\otimes 6} & M^{\otimes 7} \\      M^{\otimes 3} & M^{\otimes 4} & M^{\otimes 5} & M^{\otimes 6} \\      M^{\otimes 2} & M^{\otimes 3} & M^{\otimes 4} & M^{\otimes 5} \\      M & M^{\otimes 2} & M^{\otimes 3} & M^{\otimes 4}  \end{pmatrix} \oplus \cdots \arrow[d, "\text{induced }\mathbb{Z}/2^{n-3}\mathbb{Z}-\text{graded ring}"]                                                                                                                                                                                                                                 \\
\vdots \arrow[d, "\text{induced }\mathbb{Z}/2\mathbb{Z}-\text{graded ring}"]                                                                                                                                                                                                                                                                                                                                                                                                                                                                                                                                                                                                                                                                                                                              \\
\begin{pmatrix}    R & M & M^{\otimes 2} & M^{\otimes 3} & \cdots & M^{\otimes 2^{n-1}-1} \\    0 & R & M & M^{\otimes 2} & \cdots & M^{\otimes 2^{n-1}-2} \\    0 & 0 & R &  M & \cdots & M^{\otimes 2^{n-1}-3} \\    0 & 0 & 0 & R & \cdots & M^{\otimes 2^{n-1}-4} \\     \vdots  & \vdots  & \vdots & \vdots & \ddots & \vdots  \\    0 & 0 & 0 & 0 & \cdots & R   \end{pmatrix}\oplus\begin{pmatrix}    0 & 0 & 0 & 0 & \cdots & 0 \\    \vdots  & \vdots  & \vdots & \vdots &  & \vdots  \\    M^{\otimes 4} & M^{\otimes 5} & M^{\otimes 6} &  M^{\otimes 7} & \cdots & 0 \\    M^{\otimes 3} & M^{\otimes 4} & M^{\otimes 5} &  M^{\otimes 6} & \cdots & 0 \\    M^{\otimes 2} & M^{\otimes 3} & M^{\otimes 4} & M^{\otimes 5} & \cdots & 0 \\     M & M^{\otimes 2} & M^{\otimes 3} & M^{\otimes 4} & \cdots & 0   \end{pmatrix}
\end{tikzcd}}
\end{center}
to conclude that at each step the graded ring satisfies the Tor-vanishing condition. Consequently, the graded ring $A\oplus M'$ (which occurres at the last step above) satisfies the Tor-vanishing condition, which completes the proof. 
\end{proof}

Let $R$ be a positively and finitely graded ring over $\mathbb{Z}$ and consider the ring homomorphism $R_0\rightarrow R_0\oplus R_1\oplus \cdots$ given by $x_0\mapsto (x_0,x_1,\dots)$. By \cite[Lemma 5.2]{cummings}, using the restriction functor $\Mod R\rightarrow \Mod R_0$, it follows that if $\pd_{R_0}R_i<\infty$ for all $i$ and injectives generate for $R$, then injectives generate for $R_0$. In order to show how the theory we have developed can be applied to injective generation, we give an alternative proof of the above using covering techniques.

\begin{prop} \label{general_easy_way}
    Let $R$ be a graded ring over $\mathbb{Z}/2^n\mathbb{Z}$ such that $R_i=0$ for all $i\in\{2^{n-1},\dots,2^n-1\}$. If $\pd_{R_0}R_i<\infty$ for all $i$ and injectives generate for $R$, then injectives generate for $R_0$. 
\end{prop}
\begin{proof}
    Consider the covering ring $\Hat{R}$ and view it as the Morita context ring $\big(\begin{smallmatrix}
  A & M\\
  M & A
\end{smallmatrix}\big)_{(0,0)}$ associated to $R$. Given the hypothesis that injectives generate for $R$, it follows by Corollary \ref{injective_generation_for_covering} that injectives generate for $\Hat{R}$. Moreover, by the assumption that $\pd_{R_0}R_i<\infty$ for all $i$, it follows by Proposition \ref{reduction} (iii) that $\pd_AM<\infty$. Therefore, it follows by Theorem \ref{injectivegenerationformorita} that injectives generate for $A$, thus by Proposition \ref{reduction} (ii) it follows that injectives generate for $R_0$. 
\end{proof}

We also present the inverse to the above, in the case the bimodule $M$ in the Morita context ring is perfect. 

\begin{prop}
\label{general_hard_way}
Let $R$ be a graded ring over $\mathbb{Z}/2^n\mathbb{Z}$ such that $R_i=0$ for $i\in\{2^{n-1},\dots,2^n-1\}$. Consider the covering ring $\Hat{R}$ of $R$ and view it as the Morita context ring $\big(\begin{smallmatrix}
  A & M\\
  M & A
\end{smallmatrix}\big)_{(0,0)}$ associated to $R$. If injectives generate for $R_0$ and $M$ is a left perfect $A$-$A$-bimodule, then injectives generate for $R$. 
\end{prop}
\begin{proof}
    Since injectives generate for $R_0$, it follows by Proposition \ref{reduction} (i) that injectives generate for $A$ and since $M$ is assumed to be left perfect, it follows by Proposition~\ref{lemmataki} and Theorem~\ref{injectivegenerationformorita} that injectives generate for $\Hat{R}$. Therefore by Corollary \ref{injective_generation_for_covering} injectives generate for $R$. 
\end{proof}

In the case of tensor rings, we combine the above and prove the following. 

\begin{thm}
\label{mainthm2}
Let $R$ be a ring and $M$ a nilpotent left perfect $R$-bimodule. Then injectives generate for $R$ if and only if injectives generate for $T_R(M)$. 
\end{thm}
\begin{proof}
    Since $M$ is assumed to be nilpotent, there is an integer $k$ such that $M^{\otimes i}=0$ for all $i\geq k$. Consider $n$ such that $2^{n-1}\geq k$ and view $T_R(M)$ as the underlying ring of the graded ring $R'=\oplus_{i\in\mathbb{Z}/2^n\mathbb{Z}}R_i'$ with $R_i'=M^{\otimes i}$ where $M^{\otimes 0}:=R$. \\ 
    ($\Longleftarrow$) Assume that injectives generate for $T_R(M)$. This is the same as injectives generating for $R'$. Since $M$ is assumed to be left perfect, it follows in particular that $\pd_{R'_0}R'_i=\pd_RM^{\otimes i}<\infty$ for all $i\geq 1$ and therefore by Proposition \ref{general_easy_way} it follows that injectives generate for $R'_0=R$. \\ 
    ($\Longrightarrow$) Assume that injectives generate for $R'_0=R$ and consider the Morita context ring $\big(\begin{smallmatrix}
  A & M'\\
  M' & A
\end{smallmatrix}\big)_{(0,0)}$ associated to $R'$. By Proposition \ref{lifting_perfect_2}, we have that $M'$ is a left perfect $A$-bimodule. By Proposition \ref{general_hard_way} it follows that injectives generate for $R'$, i.e. injectives generate for $T_R(M)$.  
\end{proof}

\subsection{Perfectness, the Beilinson algebra and injective generation} Here we recall the notion of the Beilinson algebra in the sense of \cite{chen, minamoto} and compare it with the Morita context rings that we consider. 

Let $\Lambda$ be an algebra over a commutative ring. Assume that $\Lambda$ is positively and finitely graded over $\mathbb{Z}$. Fix a natural number $l$ such that $\Lambda_i=0$ for $i\geq l+1$. Then, the \emph{Beilinson algebra} associated to the pair $(\Lambda,l)$, denoted by $\mathsf{b}(\Lambda)$ and its associated $\mathsf{b}(\Lambda)$-bimodule, denoted by $\mathsf{x}(\Lambda)$, are defined as follows 
\begin{center}
    $\mathsf{b}(\Lambda)=\begin{pmatrix}
		\Lambda_0 & \Lambda_1 & \cdots & \Lambda_{l-1} \\
		0 & \Lambda_0 & \cdots & \Lambda_{l-2} \\ 
            \vdots & \vdots &  & \vdots \\ 
            0 & 0 & \cdots & \Lambda_0
	\end{pmatrix}$ and $\mathsf{x}(\Lambda)=\begin{pmatrix}
		\Lambda_l & 0 & \cdots & 0 \\
		\Lambda_{l-1} & \Lambda_l & \cdots & 0 \\ 
            \vdots & \vdots &  & \vdots \\ 
            \Lambda_1 & \Lambda_2 & \cdots & \Lambda_l
	\end{pmatrix}$
\end{center}
View $\mathsf{b}(\Lambda)\ltimes \mathsf{x}(\Lambda)$ as graded over $\mathbb{Z}$ with $\mathsf{b}(\Lambda)$ in degree 0 and $\mathsf{x}(\Lambda)$ in degree 1. By \cite{chen, minamoto}, there is an equivalence $\GrMod \Lambda\simeq \GrMod \mathsf{b}(\Lambda)\ltimes \mathsf{x}(\Lambda)$ of graded module categories. We have the following 

\begin{lem} 
\label{injgenforpositivelygraded}
Keep the notation as above. Injectives generate for $\Lambda$ if and only if graded injectives generate for $\Lambda$. 
\end{lem}
\begin{proof}
    If injectives generate for $\Lambda$, then Proposition~\ref{ungradedimpliesgraded} yields that graded injectives generate for $\Lambda$. Assume now that graded injectives generate for $\Lambda$ and consider the forgetful functor $\GrMod \Lambda\rightarrow \Mod\Lambda$. By \cite{vandenbergh} (see also \cite{yekutieli} for a proof), every graded injective $\Lambda$-module has injective dimension at most 1. Hence, the image of the functor $\mathsf{D}(\GrMod\Lambda)\rightarrow \mathsf{D}(\Lambda)$ lies in $\mathsf{Loc}(\Inj\Lambda)$, so injectives generate for $\Lambda$. 
\end{proof}

\begin{cor}
\label{cor: Beilinson algebra}
Keep the notation as above. Injectives generate for $\Lambda$ if and only if injectives generate for $\mathsf{b}(\Lambda)\ltimes\mathsf{x}(\Lambda)$. 
\end{cor}
\begin{proof}
By Lemma~\ref{injgenforpositivelygraded} it follows that injectives generate for $\Lambda$ if and only if graded injectives generate for $\Lambda$ and similarly, injectives generate for $\mathsf{b}(\Lambda)\ltimes\mathsf{x}(\Lambda)$ if and only if graded injectives generate for $\mathsf{b}(\Lambda)\ltimes\mathsf{x}(\Lambda)$. By the equivalence $\GrMod \Lambda\simeq \GrMod \mathsf{b}(\Lambda)\ltimes \mathsf{x}(\Lambda)$, it follows that graded injectives generate for $\Lambda$ if and only if graded injectives generate for $\mathsf{b}(\Lambda)\ltimes\mathsf{x}(\Lambda)$ which completes the claim. 
\end{proof}

Let now $n$ be such that $\Lambda_i=0$ for $i\geq 2^{n-1}$ and view $\Lambda$ as being graded over $\mathbb{Z}/2^n\mathbb{Z}$. Consider the covering ring $\Hat{\Lambda}$ of $\Lambda$ (with respect to the given grading) which we view as a Morita context ring $\big(\begin{smallmatrix}
  A & M\\
  M & A
\end{smallmatrix}\big)_{(0,0)}$. A comparison of the definitions, shows that $A=\mathsf{b}(\Lambda)$ where $\mathsf{b}(\Lambda)$ is the Beilinson algebra associated to the pair $(\Lambda,2^{n-1})$ and $M=\mathsf{x}(A)$. We end this section with the following remark. 

\begin{rem}
    As a byproduct of the tedious computations of Proposition \ref{lifting_perfect_2}, we get the following: If $\Lambda=T_R(M)$, then the aforementioned Proposition tells us that $\mathsf{x}(\Lambda)$ is a perfect $\mathsf{b}(\Lambda)$-bimodule (for the particular choice of Beilinson algebra that we made earlier). Therefore, the study of the graded module category of a tensor algebra with a perfect bimodule is reduced to the study of the graded module category of a trivial extension with a perfect bimodule. 
\end{rem}

\section{Cleft extensions of module categories}
\label{section:cleft extensions}

In this section we study cleft extensions of module categories with respect to injective generation.

\subsection{Cleft extensions of module categories}

We begin by recalling the definition of a cleft extension in the context of module categories. 

\begin{defn} (Beligiannis \cite{beligiannis}) \label{definition_of_cleft}
    A \emph{cleft extension} of a module category $\Mod B$ is a module category $\Mod A$ together with functors: 
    \begin{center}
        \begin{tikzcd}
\Mod B \arrow[rr, "\mathsf{i}"] &  & \Mod A \arrow[rr, "\mathsf{e}"] &  & \Mod B \arrow[ll, "\mathsf{l}"', bend right]
\end{tikzcd}
    \end{center}
henceforth denoted by $(\Mod B,\Mod A,\mathsf{i},\mathsf{e},\mathsf{l})$ such that the following hold: \\ 
(a) The functor $\mathsf{e}$ is faithful exact. \\ 
(b) The pair $(\mathsf{l},\mathsf{e})$ is an adjoint pair. \\ 
(c) There is a natural isomorphism $\mathsf{e}\mathsf{i}\simeq \mathsf{Id}_{\Mod{B}}$. 
\end{defn}

The above data induces more structural information for $(\Mod B,\Mod A,\mathsf{i},\mathsf{e},\mathsf{l})$. For instance, it can be proved that the functor $\mathsf{i}$ is fully faithful and exact. Moreover, there is a functor $\mathsf{q}\colon\Mod A\rightarrow \Mod B$ such that $(\mathsf{q},\mathsf{i})$ is an adjoint pair. Then we also obtain that $\mathsf{ql}\simeq \mathsf{Id}_B$. Moreover, there are endofunctors $\mathsf{F}\colon\Mod B\rightarrow \Mod B$ and $\mathsf{G}\colon\Mod A\rightarrow \Mod A$ that appear in the following short exact sequences: 
\begin{equation}
\label{standardexactsequences}
0\rightarrow \mathsf{G}\rightarrow \mathsf{le}\rightarrow \mathsf{Id}_{\Mod{A}}\rightarrow 0
\ \ \text{and} \ \ 
0\rightarrow \mathsf{F}\rightarrow \mathsf{el}\rightarrow \mathsf{Id}_{\Mod{B}}\rightarrow 0
\end{equation}
Note that the second one splits. The functor $\mathsf{F}$ is defined as $\mathsf{eGi}$ and from the above one can prove that $\mathsf{F}^n\mathsf{e}\simeq \mathsf{eG}^n$, thus in particular $\mathsf{F}$ is nilpotent if and only if $\mathsf{G}$ is nilpotent (see \cite[Lemma~2.4]{arrow}). An important observation that we will use later on, is that the functors $\mathsf{i}$ and $\mathsf{e}$ always preserve coproducts (see for instance \cite[Proposition~2.4 and Proposition~2.8]{beligiannis}) For a self-contained treatment of the above properties we refer the reader to \cite[Section 2]{arrow}. 

We need the following homological property of a cleft extension. 

\begin{lem} \textnormal{(\!\!\cite[Corollary 4.2]{beligiannis2})}
\label{homological_property_of_cleft} 
Let $(\Mod B,\Mod A,\mathsf{i},\mathsf{e},\mathsf{l})$ be a cleft extension of module categories and denote by $\mathsf{F}$ the associated endofunctor of $\Mod B$. Then $\mathbb{L}_i\mathsf{F}\cong \mathsf{e}\mathbb{L}_i\mathsf{l}$ for all $i\geq 1$. 
\end{lem}

\begin{rem}
\label{projcleft}
Let $X$ be an $A$-module. Consider the $B$-module $\mathsf{e}(X)$ and let $P$ be a projective $B$-module with a surjective homomorphism $P\twoheadrightarrow \mathsf{e}(X)$. Applying the functor $\mathsf{l}$, which is left adjoint to $\mathsf{e}$, gives a surjective module homomorphism $\mathsf{l}(P)\twoheadrightarrow\mathsf{le}(X)$ in $\Mod A$. Moreover, there is a surjective module homomorphism $\mathsf{le}(X)\twoheadrightarrow X$. Summing up, there is a surjective module homomorphism $\mathsf{l}(P)\twoheadrightarrow X$ in $\Mod A$ for some projective $B$-module $P$. By this observation we conclude that an $A$-module $X$ is projective if and only if it is a direct summand of $\mathsf{l}(P)$ for some projective $B$-module $P$.
\end{rem}

\subsection{Perfect endofunctors} We introduce the following notion of perfectness of an endofunctor that will be used later to detect injective generation in a cleft extension of module categories. This was also considered by Beligiannis in \cite{beligiannis2}. Our choice of the name ``perfect" is motivated by the obvious connection with the notion of a perfect bimodule (see Lemma \ref{tensor_functor_is_perfect}).

\begin{defn} 
\label{perfect_functor}
An endofunctor $\mathsf{F}\colon\Mod B\rightarrow \Mod B$ is called \emph{left perfect} if it satisfies the following conditions: 
\begin{itemize}
\item[(i)] $\mathbb{L}_i\mathsf{F}^j(\mathsf{F}(P))=0$ for every projective $B$-module $P$ and all $i,j\geq 1$. 

\item[(ii)] There is $n$ such that for every $p,q\geq 1$ with $p+q\geq n+1$, we have that $\mathbb{L}_p\mathsf{F}^q=0$.  
\end{itemize}
\end{defn}

The following lemma will be useful. 

\begin{lem} \label{equivalent_for_flat}
Let $\mathsf{F}\colon\Mod B\rightarrow \Mod B$ be a functor satisfing condition \textnormal{(i)} of the Definition~\ref{perfect_functor}. Then for an $B$-module $X$ the following conditions are equivalent: 
\begin{itemize}
\item[(i)] $\mathbb{L}_i\mathsf{F}(\mathsf{F}^j(X))=0$ for all $i\geq 1$ and all $j\geq 0$.

\item[(ii)] $\mathbb{L}_i\mathsf{F}^s(\mathsf{F}^j(X))=0$ for all $i,s\geq 1$ and all $j\geq 0$. 

\item[(iii)] $\mathbb{L}_i\mathsf{F}^j(X)=0$ for all $i,j\geq 1$.
\end{itemize}
A module $X$ is called $\mathsf{F}$-projective if it satisfies any of the above. 
\end{lem}
\begin{proof}
(i)$\implies$(ii): Consider a projective resolution of $\mathsf{F}^j(X)$ as below 
\[
\cdots\rightarrow P_n\rightarrow\cdots\rightarrow P_0\rightarrow \mathsf{F}^j(X)\rightarrow 0
\]
Since $\mathbb{L}_i\mathsf{F}(\mathsf{F}^j(X))=0$ for all $i\geq 1$, we may apply $\mathsf{F}$ to (6.5) to get an exact sequence 
\[
\cdots\rightarrow \mathsf{F}(P_n)\rightarrow \cdots\rightarrow \mathsf{F}(P_0)\rightarrow \mathsf{F}(\mathsf{F}^j(X))\rightarrow 0
\]
Since $\mathbb{L}_i\mathsf{F}(\mathsf{F}(P_k))=0$ for all $k$, applying $\mathsf{F}$ to the above exact sequence (which is the same as applying $\mathsf{F}^2$ to the first exact sequence) gives a complex whose i-th homology is $\mathbb{L}_i\mathsf{F}(\mathsf{F}^{j+1}(X))$ for $i\geq 1$. Therefore $\mathbb{L}_i\mathsf{F}^2(\mathsf{F}^j(X))\cong \mathbb{L}_i\mathsf{F}(\mathsf{F}^{j+1}(X))\cong 0$ for all $i\geq 1$. The argument for $s=3$ is the same and we proceed inductively.
    
(ii)$\implies$(iii): This implication is clear. 
    
(iii)$\implies$(i): Consider a projective resolution of $X$ as below 
\[
\cdots\rightarrow P_n\rightarrow \cdots\rightarrow P_0\rightarrow X\rightarrow 0.
\]
Since $\mathbb{L}_i\mathsf{F}^j(X)=0$ for all $i,j\geq 1$, applying $\mathsf{F}^j$ to the above gives an exact sequence 
\[
\cdots\rightarrow \mathsf{F}^j(P_n)\rightarrow \cdots\rightarrow \mathsf{F}^j(P_0)\rightarrow \mathsf{F}^j(X)\rightarrow 0.
\]
Note that condition (i) of Definition~\ref{perfect_functor} is equivalent to  $\mathbb{L}_i\mathsf{F}(\mathsf{F}^j(P))=0$ for every projective $B$-module $P$ and all $i,j\geq 1$. Since $\mathbb{L}_i\mathsf{F}(\mathsf{F}^j(P_k))=0$ for all $k$, it follows that the left derived functor $\mathbb{L}_i\mathsf{F}(\mathsf{F}^j(X))$, for $i\geq 1$, can be computed by applying $\mathsf{F}$ to the second complex. However, this is the same as the complex that we get by applying $\mathsf{F}^{j+1}$ to the first resolution of $X$ which has zero homology for $i\geq 1$, since $\mathbb{L}_i\mathsf{F}^{j+1}(X)=0$ for $i\geq 1$. 
\end{proof}

We end this section with the following proposition that will be used later on for the purposes of injective generation (compare with \cite[Theorem 7.22]{beligiannis2}). 

\begin{prop} \label{left_derived_of_q_vanishes}
    Let $(\Mod{B},\Mod{A},\mathsf{i},\mathsf{e},\mathsf{l})$ be a cleft extension of module categories such that the associated endofunctor $\mathsf{F}$ of $\Mod{B}$ is left perfect and nilpotent. Then $\mathbb{L}_n\mathsf{q}=0$ for $n$ large enough. 
\end{prop}
\begin{proof}
    Consider an object $X$ of $\Mod{A}$ such that $\mathsf{e}(X)$ is $\mathsf{F}$-projective. We notice that $\mathbb{L}_i\mathsf{q}(\mathsf{lF}^k\mathsf{e}(X))=0$ for all $i\geq 1$ and $k\geq 0$. Indeed, if $\mathsf{e}(X)$ is $\mathsf{F}$-projective, then so is $\mathsf{F}^k\mathsf{e}(X)$ for every $k\geq 0$, by Lemma \ref{equivalent_for_flat}. Therefore, it follows by Lemma \ref{homological_property_of_cleft} that $\mathbb{L}_i\mathsf{l}(\mathsf{F}^k\mathsf{e}(X))=0$ for all $i\geq 1$, which tells us that in order to compute $\mathbb{L}_i\mathsf{q}(\mathsf{lF}^k\mathsf{e}(X))$, it is enough to begin with a projective resolution of $\mathsf{F}^k\mathsf{e}(X)$ and apply $\mathsf{ql}$. But $\mathsf{ql}\simeq \mathsf{Id}_{\Mod{B}}$, which completes the first claim. Consider the following sequence of short exact sequences 
    \begin{center}
        $0\rightarrow \mathsf{G}(X)\rightarrow \mathsf{le}(X)\rightarrow X\rightarrow 0$,  \  $0\rightarrow \mathsf{G}^2(X)\rightarrow \mathsf{lFe}(X)\rightarrow \mathsf{G}(X)\rightarrow 0$,  \  $\cdots$
    \end{center}
    from which we derive the following isomorphisms for $i\geq l$:
    \begin{align*}
        \mathbb{L}_i\mathsf{q}(X)\cong\mathbb{L}_{i-1}\mathsf{q(G}(X))\cong\mathbb{L}_{i-2}\mathsf{q(G}^2(X))\cong\cdots\cong\mathbb{L}_{i-l}\mathsf{q(G}^l(X))
    \end{align*}
     for all $l\geq 0$. Since $\mathsf{F}$ is assumed to be nilpotent, so is $\mathsf{G}$, so for $s$ such that $\mathsf{G}^s=0$ and for $i\geq s-1$ we get that $\mathbb{L}_i\mathsf{q}(X)\cong \mathbb{L}_{i-s+1}\mathsf{q(G}^{s-1}(X))\cong\mathbb{L}_{i-s+1}\mathsf{q(lF}^{s-1}\mathsf{e}(X))\cong 0$. This shows that the left derived functor of $\mathsf{q}$ vanishes in degree greater or equal to $s-1$ on every object $X\in\Mod{A}$ such that $\mathsf{e}(X)$ is $\mathsf{F}$-projective. Now for an arbitrary object $X$ of $\Mod{A}$, consider an exact sequence as below 
\[
0\rightarrow X'\rightarrow P_{n-2}\rightarrow \cdots\rightarrow P_0\rightarrow X\rightarrow 0
\]
where $n$ is such that $\mathbb{L}_p\mathsf{F}^q=0$ for all $p,q\geq 1$ with $p+q\geq n+1$. If $n=1$ we set $X'=X$. By the choice of $n$, it follows that $\mathsf{e}(X')$ is $\mathsf{F}$-projective. Therefore, $\mathbb{L}_i\mathsf{q}(X')=0$ for all $i\geq s-1$, so by the exact sequence and dimension shift, it follows that $\mathbb{L}_i\mathsf{q}(X)=0$ for all $i\geq n+s-2$. 
\end{proof}

\subsection{Cleft extensions of module categories are $\theta$-extensions} In this section we recall the construction of $\theta$-extensions of rings due to Marmaridis \cite{marmaridis} and a theorem of Beligiannis stating that cleft extensions of module categories occur precisely as $\theta$-extensions of rings. 

\begin{defn} \textnormal{(\!\!\cite{marmaridis})}
Let $R$ be a ring, $M$ an $R$-bimodule and $\theta\colon M\otimes_RM\rightarrow M$ an associative $R$-bimodule homomorphism, i.e a bimodule homomorphism such that $\theta(\theta\otimes \mathsf{Id}_M)=\theta(\mathsf{Id}_M\otimes\theta)$. Then the \emph{$\theta$-extension of $R$ by $M$}, denoted by $R\ltimes_{\theta}M$, is the ring with underlying group $R\oplus M$ and multiplication given as follows 
\[
(r,m)\cdot (r',m')=(rr',rm'+mr'+\theta(m\otimes m'))
\]
for all $r,r'\in R$ and $m,m'\in M$. 
\end{defn}

\begin{exmp}  
\label{examples_of_theta}
We list here some examples of $\theta$-extension of rings.
\begin{itemize}
\item[(i)] Given a ring $R$ and an $R$-bimodule $M$, consider $\theta\colon M\otimes_RM\rightarrow M$ to be the zero homomorphism. Then the $\theta$-extension $R\ltimes_{\theta}M$ is exactly the usual trivial extension $R\ltimes M$. 

\item[(ii)] Let $R$ be a ring $R$ and $M$ an $R$-bimodule. Consider the $R$-bimodule $M'=M\oplus M^{\otimes 2}\oplus\cdots$ and $\theta\colon M'\otimes_RM'\rightarrow M'$ the $R$-bimodule homomorphism induced by the natural homomorphisms $M^{\otimes k}\otimes_RM^{\otimes l}\rightarrow M^{\otimes k+l}$. Then the $\theta$-extension $R\ltimes_{\theta}M'$ is isomorphic to the tensor ring $T_R(M)$.
 
\item[(iii)] Let $\Lambda$ be a positively graded ring over the integers. Consider the {$\Lambda_0$-bimodule} $M:=\Lambda_1\oplus \Lambda_2\oplus \cdots$ and $\theta\colon M\otimes_{\Lambda_0}M\rightarrow M$ the $\Lambda_0$-bimodule homomorphism induced by multiplication in $\Lambda$. Then, there is an isomorphism of rings $\Lambda\cong \Lambda_{0}\ltimes_{\theta}M$.
\end{itemize}
\end{exmp}

Let $R\ltimes_{\theta}M$ be a $\theta$-extension. We have the following ring homomorphisms: $R\rightarrow R\ltimes_{\theta}M$ given by $r\mapsto (r,0)$ and $R\ltimes_{\theta}M\rightarrow R$ given by $(r,m)\mapsto r$. Then, we get the following diagram 
\[
\begin{tikzcd} 
\Mod R \arrow[rr, "\mathsf{Z}"] &  & \Mod R\ltimes_{\theta}M \arrow[rr, "\mathsf{U}"] \arrow[ll, "\mathsf{C}"', bend right] &  & \Mod R \arrow[ll, "\mathsf{T}"', bend right]
\end{tikzcd}
\]
where $\mathsf{U}$ and $\mathsf{Z}$ are restriction functors induced by the latter homomorphisms and the left adjoints are $\mathsf{T}=-\otimes_R R\ltimes_{\theta}M$ and $\mathsf{C}=-\otimes_{R\ltimes_{\theta}M}R$. Also, the associated endofunctor on $\Mod R$ is given by $-\otimes_{R}M$.

The above diagram is a cleft extension of module categories. In fact, every cleft extension occurs as a $\theta$-extension, like above. This is due to Beligiannis and we spell it out in the following proposition. 

\begin{prop} \textnormal{(\!\!\cite[Theorem 2.6, Proposition 3.3]{beligiannis})}
\label{equivalence_of_cleft}
Let $(\Mod B,\Mod A,\mathsf{i},\mathsf{e},\mathsf{l})$ be a cleft extension of module categories. Then, there exists a $B$-bimodule $M$, an associative $B$-bimodule homomorphism $\theta\colon M\otimes_{B}M\rightarrow M$ and an equivalence $\Mod A\rightarrow \Mod B\ltimes_{\theta}M$ making the following diagram commutative$\colon$
\[
\begin{tikzcd}
\Mod B \arrow[rr, "\mathsf{i}"] \arrow[dd, "\mathsf{Id}_{\Mod B}"] &  & \Mod A \arrow[rr, "\mathsf{e}"] \arrow[ll, "\mathsf{q}"', bend right] \arrow[dd, "\simeq"] &  & \Mod B \arrow[ll, "\mathsf{l}"', bend right] \arrow[dd, "\mathsf{Id}_{\Mod B}"] \\
                                                                   &  &                                                                                            &  &                                                                                 \\
\Mod B \arrow[rr, "\mathsf{Z}"]                                    &  & \Mod B\ltimes_{\theta}M \arrow[rr, "\mathsf{U}"] \arrow[ll, "\mathsf{C}"', bend right]     &  & \Mod B \arrow[ll, "\mathsf{T}"', bend right]                                   
\end{tikzcd}
\]
\end{prop}

\subsection{Cleft extensions and injective generation} 
In this subsection we investigate injective generation for cleft extensions of module categories using the machinery developed in the previous subsections. Let $(\Mod B,\Mod A,\mathsf{i},\mathsf{e},\mathsf{l})$ be a cleft extension of module categories. Consider the endofunctor $\mathsf{R}\colon\Mod A\rightarrow \Mod A$ that appears as the kernel of the unit map of the adjunction $(\mathsf{q},\mathsf{i})$:
\[
0\rightarrow \mathsf{R}\rightarrow \mathsf{Id}_{\Mod{A}}\rightarrow\mathsf{iq}\rightarrow 0
\] 

For the endofunctor $\mathsf{R}$ we need the following auxiliary result. 

\begin{lem} \label{F_nilpotent_implies_R_nilpotent}
    Let $(\Mod B,\Mod A,\mathsf{i},\mathsf{e},\mathsf{l})$ be a cleft extension of module categories. If the functor $\mathsf{F}$ is nilpotent, then so is $\mathsf{R}$. 
\end{lem}
\begin{proof} 
Under the equivalence $\mathsf{\Omega}$ of Proposition \ref{equivalence_of_cleft} we see that $\mathsf{F}'=\mathsf{\Omega F \Omega}^{-1}$ and $\mathsf{R}'=\mathsf{\Omega R\Omega}^{-1}$ where $\mathsf{F}'=-\otimes_RM$ and $\mathsf{R}'\colon\Mod B\ltimes_{\theta}M\rightarrow \Mod B\ltimes_{\theta}M$ is the endofunctor that appears as the kernel of the adjunction ($\mathsf{C},\mathsf{Z}$). Therefore the functor $\mathsf{F}$ is nilpotent if and only if $M$ is nilpotent and $\mathsf{R}$ is nilpotent if and only if $\mathsf{R}'$ is nilpotent. We then have the following implications (where we write ``nil." for nilpotent)
\[
\mathsf{F} \text{ nil.} \iff M \text{ nil.} \implies \theta \text{ nil.} \stackrel{(*)}{\iff} \mathsf{R}' \text{ nil.}\iff \mathsf{R} \text{ nil.}
\]
where $(*)$ follows from \cite[Proposition 7.4 (i)]{beligiannis} and the rest implications are clear. 
\end{proof}

By the above and in preparation for injective generation, we present the following. 

\begin{lem} \label{image_of_i}
    Let $(\Mod B,\Mod A,\mathsf{i},\mathsf{e},\mathsf{l})$ be a cleft extension of module categories. If the endofunctor $\mathsf{F}$ is nilpotent, then $\mathsf{Loc(Imi)}=\mathsf{D}(A)$. 
\end{lem}
\begin{proof}
    Let $X$ be an $A$-module and consider the following short exact sequences 
    \begin{center}
        $0\rightarrow \mathsf{R}(X)\rightarrow X\rightarrow \mathsf{iq}(X)\rightarrow 0$, \  $0\rightarrow \mathsf{R}^2(X)\rightarrow \mathsf{R}(X)\rightarrow \mathsf{iqR}(X)\rightarrow 0$, \  $\cdots$
    \end{center}
    The rightmost terms of the above short exact sequences are in $\mathsf{Imi}$. Since $\mathsf{F}$ is nilpotent, so is $\mathsf{R}$ by Lemma \ref{F_nilpotent_implies_R_nilpotent}. Therefore we conclude that $X\in\mathsf{Loc(Imi)}$ which completes the claim.
\end{proof}

We can now detect injective generation in a cleft extension. 

\begin{thm} 
\label{injective_generation_for_cleft}
Let $(\Mod B,\Mod A,\mathsf{i},\mathsf{e},\mathsf{l})$ be a cleft extension of module categories. If the endofunctor $\mathsf{F}$ is left perfect and nilpotent, then injectives generate for $B$ if and only if injectives generate for $A$. 
\end{thm}
\begin{proof}
    $(\Longleftarrow)$ Assume that injectives generate for $A$. Since $\mathsf{F}$ is assumed to be left perfect, it follows by Lemma \ref{homological_property_of_cleft} that $\mathbb{L}_n\mathsf{l}=0$ for $n$ large enough and therefore $\mathbb{L}\mathsf{l}\colon\mathsf{D}(B)\rightarrow\mathsf{D}(A)$ maps complexes bounded in cohomology to complexes bounded in cohomology. By Remark \ref{bounded_cohomology} and the fact that $\mathsf{e}\colon\mathsf{D}(A)\rightarrow\mathsf{D}(B)$ preserves coproducts, we conclude that $\mathsf{Loc}(\mathsf{Ime})\subseteq \mathsf{Loc}(\Inj B)$. Since $\mathsf{e}$ is essentially surjective (on the level of module categories), we have that $\mathsf{Loc}(\mathsf{Ime})=\mathsf{D}(B)$ which completes the claim. \\ 
    $(\Longrightarrow)$ Assume that injectives generate for $B$. Since $\mathsf{F}$ is assumed to be perfect and nilpotent, it follows by Proposition \ref{left_derived_of_q_vanishes} that $\mathbb{L}_n\mathsf{q}=0$ for $n$ large enough and therefore $\mathbb{L}\mathsf{q}\colon \mathsf{D}(A)\rightarrow \mathsf{D}(B)$ maps complexes bounded in cohomology to complexes boudned in cohomology.  By Remark \ref{bounded_cohomology} and the fact that $\mathsf{i}\colon\mathsf{D}(B)\rightarrow\mathsf{D}(A)$ preserves coproducts, we conclude that $\mathsf{Loc}(\mathsf{Imi})\subseteq\mathsf{Loc}(\Inj A)$. However, since $\mathsf{F}$ is assumed to be nilpotent, it follows by Lemma \ref{image_of_i} that $\mathsf{Loc}(\mathsf{Imi})=\mathsf{D}(A)$ which completes the claim. 
\end{proof}

In order to apply Theorem \ref{injective_generation_for_cleft} to $\theta$-extensions, we need the next observation.

\begin{lem} 
\label{tensor_functor_is_perfect}
Let $R$ be a ring and $M$ an $R$-bimodule. If $M$ is left perfect and nilpotent then the functor $-\otimes_RM\colon \Mod R\rightarrow \Mod R$ is left perfect and nilpotent. 
\end{lem}

We have the following main result regarding injective generation of $\theta$-extensions, which is Theorem~C of the Introduction. 

\begin{cor} 
\label{mainthm3}
Let $R$ be a ring, $M$ an $R$-bimodule and $\theta\colon M\otimes_RM\rightarrow M$ an associative $R$-bimodule homomorphism. If $M$ is left perfect and nilpotent, then injectives generate for $R$ if and only if injectives generate for $R\ltimes_{\theta}M$. 
\end{cor}
\begin{proof}
    As explained, we have a cleft extension diagram $(\Mod R,\Mod R\ltimes_{\theta}M,\mathsf{i},\mathsf{e},\mathsf{l})$. Necessarily $\mathsf{F}\simeq -\otimes_RM$ and since $M$ is assumed to be perfect and nilpotent, the functor $\mathsf{F}$ is perfect and nilpotent by Lemma \ref{tensor_functor_is_perfect}, therefore Theorem \ref{injective_generation_for_cleft} applies. 
\end{proof}

We conclude this subsection with the following two corollaries, which are special cases of the above. 

\begin{cor} 
\label{positively_graded}
Let $\Lambda$ be a positively graded ring over $\mathbb{Z}$. If each $\Lambda_i$ is nilpotent and left perfect as a $\Lambda_0$-bimodule, then injectives generate for $\Lambda_0$ if and only if injectives generate for $\Lambda$. 
\end{cor}
\begin{proof}
 In this case $\Lambda\cong \Lambda_0\ltimes_{\theta}M$ where $M=\Lambda_1\oplus \Lambda_2\oplus\cdots$ and $\theta\colon M\otimes_{\Lambda_0}M\rightarrow M$ is given by multiplication. Then for $M$ to be nilpotent and left perfect, it is enough that every $\Lambda_i$ is nilpotent and left perfect, so the claim follows by Corollary~\ref{mainthm3}. 
\end{proof}

In the following Corollary we recover Theorem~\ref{mainthm2}.

\begin{cor}
Let $R$ be a ring and $M$ an $R$-bimodule. If $M$ is nilpotent and left perfect, then injectives generate for $R$ if and only if injctives generate for $T_R(M)$. 
\end{cor}
\begin{proof}
By Corollary \ref{positively_graded}, it is enough that for every $i\geq 1$, the $R$-bimodule $M^{\otimes i}$ is left perfect and nilpotent, which is true by Corollary~\ref{corollary_for_perfect} (iii). 
\end{proof}

\section{Twisted tensor products of algebras}
\label{section:twisted tensor products} 

Let $A$ and $B$ be two finite dimensional algebras over a field $k$. By \cite[Proposition 3.2]{cummings}, if injectives generate for $A$ and $B$, then injectives generate for the tensor product algebra $A\otimes_k B$. The purpose of this section is to extend this result to twisted tensor products in the sense of Bergh and Oppermann \cite{twisted}.

\subsection{Injective generation for graded algebras}

Let $A$ be a finite dimensional algebra over a field $k$ that is also graded over an abelian group $\Gamma$. We first note that the opposite algebra $A^{op}$ is also graded with $A^{op}_{\gamma}:=A_{-\gamma}$. Moreover, if we denote by $\mathsf{D}$ the usual duality $\mathsf{Hom}_k(-,k)\colon\smod A\rightarrow \smod A^{op}$, then this induces a duality $\mathsf{D}\colon \Grmod A\rightarrow \Grmod A^{op}$ in the following way: for a graded $A$-module $X$, the $A^{op}$-module $\mathsf{D}(X)$ is a graded $A^{op}$-module with grading given by $\mathsf{D}(X)_{\gamma}:=\mathsf{D}(X_{-\gamma})$. This is proved in \cite{artinalgebras} in the case $\Gamma=\mathbb{Z}$ but the same arguments work in full generality (see also the discussion in \cite[Section 2]{yamaura}). 

For a finite dimensional algebra $A$ over a field, it is well-known that the injective modules are direct summands of a coproduct of copies of $\mathsf{D}(A)$. The analogous holds (and should be well-known to experts) in the graded case, but we could not find a reference in the literature. For this reason, we begin with the following. 

\begin{lem}
Let $A$ be a finite dimensional algebra that is also graded over an abelian group $\Gamma$. Then every graded injective module is a direct sum of finitely generated graded injective modules. 
\end{lem}
\begin{proof}
Let $M$ be a graded $A$-module. Consider its graded socle, $\mathsf{soc}^{gr}(M)$ which, by definition, is the direct sum of all the graded simple submodules of $M$, say $S_j, j\in J$. We make some key observations.  

(i) Since $A$ is finite dimensional, it follows that $\mathsf{D}(A)=\mathsf{Hom}_k(A,k)$ is finitely generated. Further, $\mathsf{D}(A)$ is graded and injective, thus graded injective. 

(ii) Let $S$ be a graded simple $A$-module and denote by $\mathsf{E}^{gr}(S)$ its graded injective envelope and by $\mathsf{E}(S)$ its injective envelope. Then $\mathsf{E}^{gr}(S)$ is a submodule of $\mathsf{E}(S)$, which is finitely generated as $S$ can be embeded into the finitely generated module $\mathsf{Hom}_k(A,k)$. Therefore $\mathsf{E}^{gr}(S)$ is also finitely generated. 

Consider the graded module $I=\oplus_{j\in J} \mathsf{E}^{gr}(S_j)$, which is graded injective. The graded ring homomorphisms $\mathsf{soc}^{gr}(M)\rightarrow M$ and $\mathsf{soc}^{gr}(M)\rightarrow I$ can be completed to a commutative triangle as below
\begin{center}
	\begin{tikzcd}
		\mathsf{soc}^{gr}(M) \arrow[r] \arrow[d] & M \arrow[ld, dashed] \\
		I                               &                     
	\end{tikzcd}
\end{center}
The homomorphism $M\rightarrow I$ is an essential monomorphism (which follows from the fact that both $\mathsf{soc}^{gr}(M)\rightarrow M$ and $\mathsf{soc}^{gr}(M)\rightarrow I$ are essential monomorphisms) and so $I=\mathsf{E}^{gr}(M)$. If $M$ is injective, then $M=I$ from which the result follows. 
\end{proof}

By the above, together with the graded duality, we infer the following.

\begin{cor} \label{structure_of_graded_injectives}
     A graded $A$-module $I$ is graded injective if and only if it is a direct summand of a coproduct of copies of $\mathsf{D}(A)(\gamma)$ for various $\gamma\in\Gamma$. In particular a graded $A$-module is graded injective if and only if it is injective. 
\end{cor}

\begin{cor} \label{comparison_for_graded_algebras}
    Graded injectives generate for $A$ if and only if injectives generate for $A$. 
\end{cor}
\begin{proof}
One implication is always true, even for rings (see Proposition \ref{ungradedimpliesgraded}). By Corollary \ref{structure_of_graded_injectives} it follows that the functor $\GrMod A\rightarrow \Mod A$ maps graded injective modules to injective modules, from which the inverse implication follows (verbatim to the proof of Proposition \ref{finitegrading}).
\end{proof}

\subsection{Injective generation for twisted tensor products}
Let us now recall the definition of twisted tensor products. Let $A$ and $B$ be finite dimensional algebras over a field $k$ that are also graded over abelian groups $\Gamma_1$ and $\Gamma_2$ respectively. Moreover, fix a bicharacter $t\colon\Gamma_1\times\Gamma_2\rightarrow k^{\times}$, i.e a homomorphism of abelian groups such that
\begin{itemize}
\item[(i)] $t(0,\gamma_2)=t(\gamma_1,0)=1$ 
\item[(ii)] $t(\gamma_1+\gamma_1',\gamma_2)=t(\gamma_1,\gamma_2)t(\gamma_1',\gamma_2)$ 
\item[(iii)]$t(\gamma_1,\gamma_2+\gamma_2')=t(\gamma_1,\gamma_2)t(\gamma_1,\gamma_2')$ 
\end{itemize}
Recall that for a graded algebra $A$ and a homogeneous element $a\in A$, we write $|a|$ for the degree of $a$.

\begin{defn} \textnormal{(\!\!\cite[Definition 2.2]{twisted})}
The \emph{twisted tensor product algebra} of $A$ and $B$, denoted by $A\otimes^tB$ is defined to be $A\otimes_k B$ as a vector space, graded over $\Gamma_1\times \Gamma_2$, with multiplication given by 
\begin{center}
	$(a\otimes b)\cdot (a'\otimes b')=t(|a'|,|b|)(aa'\otimes bb')$
\end{center}
for homogeneous $a,a'\in A$ and $b,b'\in B$. 
\end{defn}
Let $M$ be a right graded $A$-module and $N$ a right graded $B$-module. Consider $M\otimes_kN$ which we endow with the structure of a graded $A\otimes^tB$-module with action given by 
\begin{center}
$(m\otimes n)(a\otimes b)=t(|a|,|n|)(ma\otimes nb)$
\end{center}
for homogeneous $m\in M, n\in N,a\in A,b\in B$. The resulting $A\otimes^tB$-module is denoted by $M\otimes^tN$. Given $\gamma\in\Gamma_1$ and $\gamma'\in\Gamma_2$, we denote by $(M\otimes^tN)(\gamma,\gamma')$ the twist of $M\otimes^{t}N$ by $(\gamma,\gamma')\in\Gamma_1\times \Gamma_2$. We fix the above notation and recall the following elementary properties.
\begin{lem} 
\label{basic_properties_of_modules_over_twisted}
The following statements hold. 
\begin{itemize}
\item[(i)] Let $\{M_i \ | \ i\in I\}$ be a family of graded $A$-modules and $N$ a graded $B$-module. There is an isomorphism 
\[
(\oplus_i M_i)\otimes^tN\cong \oplus_i(M_i\otimes^tN)
\]
of graded $A\otimes^tB$-modules.  

\item[(ii)] There is an isomorphism
\[
\mathsf{Hom}_k(A\otimes^tB,k)\cong \mathsf{Hom}_k(A,k)\otimes^t\mathsf{Hom}_k(B,k)
\]
of graded ${A\otimes^tB}$-modules. 

\item[(iii)] Given a graded $A$-module $M$ and a graded $B$-module $N$, there is an isomorphism
\[
M(\gamma)\otimes^tN(\gamma')\cong (M\otimes^tN)(\gamma,\gamma')
\]
of graded $A\otimes^tB$-modules, for all $\gamma\in\Gamma_1$ and $\gamma'\in\Gamma_2$. 
\end{itemize}
\end{lem}
\begin{proof}
(i) We have an isomorphism of groups 
\begin{center}
	$(\oplus_iM_i)\otimes^tN\cong \oplus_i(M_i\otimes^tN)$
\end{center}
and it is straightforward to check that it respects the action of $A\otimes^tB$. 
 
(ii) We have the following isomorphisms of groups as $A$ and $B$ are assumed to be finite dimensional.
\begin{align*}
    \mathsf{Hom}_k(A\otimes^tB,k)&=\mathsf{Hom}_k(A\otimes_kB,k) \\ 
                                 &\cong \mathsf{Hom}_k(A,k)\otimes \mathsf{Hom}_k(B,k) \\ 
                                 &\cong \mathsf{Hom}_k(A,k)\otimes^t\mathsf{Hom}_k(B,k)
\end{align*}
As in (i), the above isomorphisms of groups respect the action of ${A\otimes^tB}$. 
 
(iii) The isomorphism is given by $m\otimes n\mapsto t(|\gamma|,|\gamma'|)m\otimes n$ for homogeneous $m\in M$ and $n\in N$. 
\end{proof} 

\begin{cor} \label{twisted_tensor_product_of_injectives}
Let $M$ be injective as an $A$-module and $N$ be injective as a $B$-module. If $A$ and $B$ are graded over abelian groups $\Gamma_1$ and $\Gamma_2$ respectively, then $M\otimes ^tN$ is injective as an $A\otimes^tB$-module. 
\end{cor}
\begin{proof}
By combining (i), (ii) and (iii) of Lemma \ref{basic_properties_of_modules_over_twisted}, the fact that a graded module over $A$ (respectively $B$) is graded injective if and only if it is a direct summand of a direct sum of copies of $\mathsf{D}A(\gamma)$ for all $\gamma\in \Gamma_1$ (respectively of copies $\mathsf{D}B(\gamma')$ for all $\gamma'\in \Gamma_2$) and the fact that $\mathsf{D}A(\gamma)\cong (\mathsf{D}A)(\gamma)$ for every $\gamma$ (and the analogous for $B$), it is enough to consider the case $M=\mathsf{D}A$ and $N=\mathsf{D}B$ in which case 
\begin{center}
	$\mathsf{D}A\otimes^t\mathsf{D}B\cong \mathsf{D}(A\otimes^tB)$
\end{center}
which completes the proof. 
\end{proof}

\begin{prop}
The twisted tensor product induces additive exact functors 
\begin{center}
	$-\otimes^tN\colon\GrMod A\rightarrow \GrMod(A\otimes^tB)$ 
\end{center}
and 
\begin{center}
	$M\otimes^t-\colon\GrMod B\rightarrow \GrMod(A\otimes^tB)$
\end{center}
\end{prop}
Let us first explain the above and then proceed to sketch a proof. As the situation is similar for both functors, we shall only explain it for $-\otimes^tN$. First of all, a graded $A$-module $M$ is mapped to $M\otimes^tN$, which is a graded $A\otimes^tB$-module, as we explained early in this section. Furthermore, a homomorphism $f\colon M\rightarrow M'$ of graded $A$-modules is mapped to $f\otimes^tN$ which is given by $m\otimes n\mapsto f(m)\otimes n$ for homogeneous $m\in M$ and $n\in N$. This is indeed a homomorphism between the graded $A\otimes^tB$-modules $M\otimes^tN$ and $M'\otimes^tN$, as for homogeneous $a\in A$, $b\in B$, $m\in M$ and $n\in N$ we have 
\begin{align*}
    (f\otimes^tN)((m\otimes n)(a\otimes b))&=(f\otimes^tN)(t(|n|,|a|)(ma,nb))\\ 
    &=t(|n|,|a|)f(ma)\otimes nb \\ 
    &=(f(m)\otimes n)(a\otimes b) \\ 
    &=((f\otimes^tN)(m\otimes n))(a\otimes b)
\end{align*}

\begin{proof}
The fact that it respects identities and compositions is easily seen, by similar arguments as above. The fact that the functors are additive follows from Lemma \ref{basic_properties_of_modules_over_twisted}. Lastly, for exactness we may work only with the underlying graded vector spaces, which is well known. We leave the details to the reader. 
\end{proof}

From the above, given a right graded $A$-module $M$ and a right graded $B$-module $N$, we may consider triangulated functors 
\begin{center}
$F_N\colon\mathsf{D}(\GrMod A)\rightarrow \mathsf{D}(\GrMod(A\otimes^tB))$ 
\end{center}
and
\begin{center}
$G_M\colon \mathsf{D}(\GrMod B)\rightarrow \mathsf{D}(\GrMod(A\otimes^tB))$
\end{center}

\begin{prop}
\label{graded_injective_generation_for_twisted}
Let $A$ and $B$ be finite dimensional algebras over a field $k$ that are also graded over abelian groups $\Gamma_1$ and $\Gamma_2$ respectively. If graded injectives generate for $A$ and $B$, then graded injectives generate for $A\otimes^tB$.
\end{prop}
\begin{proof}
Since graded injectives generate for $A$ and $B$, we have 
\begin{center}
	$\mathsf{D}(\GrMod A)=\mathsf{Loc}( \mathsf{D}A(\gamma),\gamma\in\Gamma_1)$
\end{center} 
and 
\begin{center}
	$\mathsf{D}(\GrMod B)=\mathsf{Loc}( \mathsf{D}B(\gamma'),\gamma'\in\Gamma_2)$
\end{center}
That is because, every injective object in, say, $\GrMod A$ is a direct summmand of a direct sum of copies of $\mathsf{D}A(\gamma)$, for various $\gamma\in\Gamma_1$ and thus $\mathsf{Loc}( \GrInj A)=\mathsf{Loc}( \mathsf{D}A(\gamma),\gamma\in\Gamma_1)$.
We note that $F_{\mathsf{D}B(\gamma')}$ satisfies the conditions of Proposition \ref{image_is_contained_in_localizing} as for every $S\in \GrInj A$ we have $F_{\mathsf{D}B(\gamma')}(S)=S\otimes^t\mathsf{D}B(\gamma')$ which is graded injective over $A\otimes^tB$. In particular, 
\begin{center}
	$F_{\mathsf{D}B(\gamma')}(A(\gamma))=A(\gamma)\otimes^t\mathsf{D}B(\gamma')=G_{A(\gamma)}(\mathsf{D}B(\gamma'))$
\end{center}
is in $\GrInj (A\otimes^tB)$. Therefore, $G_{A(\gamma)}$ satisfies the conditions of Proposition \ref{image_is_contained_in_localizing} and in particular we get that 
\begin{center}
	$G_{A(\gamma)}(B(\gamma'))=A(\gamma)\otimes^tB(\gamma')$
\end{center}
is in $\mathsf{Loc}( \GrInj (A\otimes^tB))$ for all $\gamma\in \Gamma_1$ and $\gamma'\in\Gamma_2$. Since $\mathsf{D}(\GrMod (A\otimes^tB))$ is generated by $\{(A\otimes^tB)(\gamma,\gamma') \ | \ (\gamma,\gamma')\in\Gamma_1\times\Gamma_2\}$ and $A(\gamma)\otimes^t B(\gamma')\cong (A\otimes^tB)(\gamma,\gamma')$, the result follows. 
\end{proof}

\begin{thm}
\label{injective_generation_for_twisted}
Let $A$ and $B$ be finite dimensional algebras over a field that are also graded over abelian groups. If injectives generate for $A$ and $B$, then injectives generate for $A\otimes^tB$. 
\end{thm}
\begin{proof}
The result follows from Proposition~\ref{graded_injective_generation_for_twisted} and Corollary~\ref{comparison_for_graded_algebras}. 
\end{proof}

\subsection{What is not known}
We close this section with a discussion of several natural questions on injective generation both in the ungraded and in the graded setting. 
For a list of open problems on injective generation we refer the reader to Rickard's paper \cite[Section 8]{rickard}. Let us begin with the following question.

\begin{question1}
    Let $A$ be a ring and $M$ an $A$-bimodule. Does injective generation for $A$ imply injective generation for $\big(\begin{smallmatrix}
  A & M\\
  M & A
\end{smallmatrix}\big)_{(0,0)}$? 
\end{question1}

This is not clear from what we have proved in this paper. However, a positive answer to the above question would imply a positive answer to the following problem, which is the guiding problem in the graded theory of injective generation, as explained in the Introduction.

\begin{problem}
Let $R=\bigoplus_{i\in \mathbb{Z}} R_i$ be a positively and finitely graded ring over the integers. If injectives generate for $R_0$, does it follow that injectives generate for $R$?
\end{problem}

Let us explain why answering Question~1 will provide a positive answer to the above problem. Let $R=\bigoplus_{i\in \mathbb{Z}} R_i$ be a positively and finitely graded ring over $\mathbb{Z}$. We view $R$ as a graded ring over $\mathbb{Z}/2^n\mathbb{Z}$ where $n$ is so that $R_i=0$ for all $i\in\{2^{n-1},\dots,2^n-1\}$. Consider the covering ring $\Hat{R}$ with respect to the latter grading and view it as a Morita context ring $\big(\begin{smallmatrix}
  A & M\\
  M & A
\end{smallmatrix}\big)_{(0,0)}$ with zero bimodule homomorphisms, see Lemma~\ref{lem:coveringzerozero}. Then, injective generation for $R_0$ implies injective generation for $A$, which (given a positive answer to the above question) implies injective generation for $\Hat{R}$. The latter implies injective generation for $R$. 

As a particular case, we may assume that $M\otimes_A M=0$. So we can reformulate Question~1 as follows: 

\begin{question2}
Let $A$ be a ring and $M$ an $A$-bimodule such that $M\otimes_A M=0$. Does injective generation for $A$ imply injective generation for $\big(\begin{smallmatrix}
  A & M\\
  M & A
\end{smallmatrix}\big)_{(0,0)}$? 
\end{question2}

A positive answer to Question 2 would imply a positive answer to Question 1. Indeed, assume the setup of Question 1 and write $\big(\begin{smallmatrix}
  A & M\\
  M & A
\end{smallmatrix}\big)_{(0,0)}$ as a trivial extension $\Gamma\ltimes N$, where $\Gamma=A\times A$ and $M=A\oplus A$ (see \cite[Proposition~2.5]{morita}). Then, view $\Gamma\ltimes N$ as a graded ring over $\mathbb{Z}/4\mathbb{Z}$ and consider its covering ring as follows:
\begin{center}
    $\begin{pmatrix}
  \Gamma & N & 0 & 0  \\
  0 & \Gamma & N & 0 \\ 
  0 & 0 & \Gamma & N \\ 
  N & 0 & 0 & \Gamma
\end{pmatrix}=\begin{pmatrix}
    B & M' \\ 
    M' & B
\end{pmatrix}_{(0,0)}$
\end{center}
Injective generation for $A$ implies injective generation for $\Gamma$ which by Corollary~\ref{injective_generation_for_triangular_matrix_rings} implies injective generation for the triangular ring $B$. Furthermore, $M\otimes_A M=0$ yields that $M'\otimes_BM'=0$ and therefore (assuming that the answer to Question 2 is positive) injectives generate for the covering ring of $\Gamma\ltimes N$. It follows by Corollary \ref{injective_generation_for_covering} that injectives generate for $\Gamma\ltimes N=\big(\begin{smallmatrix}
  A & M\\
  M & A
\end{smallmatrix}\big)_{(0,0)}$.

By employing Proposition \ref{reduction}, we can adapt the above Problem with the assumptions $\pd {_{R_0}R_i}<\infty$ (or $\pd {R_i}_{R_0}<\infty$) and still reduce it to the analogous of Question~1 with the assumption that $\pd _AM<\infty$ (or $\pd M_A<\infty$). We can also restrict to finite dimensional algebras. So, let us state a final question. 

\begin{question3}
    Let $A$ be a finite dimensional algebra over a field. Let $M$ be a finitely generated $A$-bimodule such that $M\otimes_A M=0$. If $\gd A<\infty$, do injectives generate for $\big(\begin{smallmatrix}
  A & M\\
  M & A
\end{smallmatrix}\big)_{(0,0)}$? 
\end{question3}

Let $\Lambda$ be a finite dimensional algebra over a field that is also graded over the integers. Assume that the global dimension of $\Lambda_0$ is finite. Then a positive answer to the above question (and in view of Rickard's celebrated theorem \cite[Theorem~3.5]{rickard}) would imply that the (big) finitistic dimension of $\Lambda$ is finite.

\end{document}